\documentclass[12pt,reqno]{amsart}

\usepackage{amsfonts,amssymb}
\usepackage{times,amsmath,amstext,amsbsy,amsopn,amsthm,upref,amscd,eucal}
\usepackage[all]{xy}
\usepackage{enumerate}

\allowdisplaybreaks

\newtheorem{theorem}{Theorem}[section]
\newtheorem{Luna Slice Theorem}[theorem]{Luna Slice Theorem}

\newtheorem{corollary}[theorem]{Corollary}
\newtheorem{lemma}[theorem]{Lemma}
\newtheorem{proposition}[theorem]{Proposition}
\newtheorem{proposition-definition}[theorem]{Proposition-Definition}

\theoremstyle{definition}
\newtheorem{definition}[theorem]{Definition}

\theoremstyle{remark}
\newtheorem{remark}[theorem]{Remark}

\newcommand{\fatdot}{{\scriptscriptstyle \bullet}}

\newcommand{\rk}{\operatorname{rk}\nolimits}

\newcommand{\GL} {\operatorname{GL}\nolimits}
\newcommand{\End} {\operatorname{End}\nolimits}

\newcommand{\Def} {\operatorname{Def}\nolimits}
\newcommand{\Art} {\operatorname{Art}\nolimits}
\newcommand{\Set} {\operatorname{Set}\nolimits}

\newcommand\CC{{\mathbb C}}

\newcommand\HH{{\mathbb H}}

\newcommand{\RRR}{{\mathcal R}}
\newcommand{\BBB}{{\mathcal B}}
\newcommand{\GGG}{{\mathcal G}}

\newcommand{\LLL}{{\mathcal L}}
\newcommand{\DDD}{{\mathcal D}}
\newcommand{\FFF}{{\mathcal F}}
\newcommand{\KKK}{{\mathcal K}}
\newcommand{\AAA}{{\mathcal A}}
\newcommand{\MMM}{{\mathcal M}}
\newcommand{\CCC}{{\mathcal C}}

\newcommand{\OOO}{{\mathcal O}}
\newcommand{\EEE}{{\mathcal E}}

\renewcommand\phi{\varphi}

\newcommand{\At}{\operatorname{At}\nolimits}

\newcommand{\ud}{\mathrm{d}}

\newcommand{\ad}{\operatorname{ad}\nolimits}

\newcommand{\Spec}{\operatorname{Spec}\nolimits}
\newcommand{\ob}{\operatorname{ob}\nolimits}

\newcommand{\im}{\operatorname{im}\nolimits}
\newcommand{\id}{\operatorname{id}\nolimits}

\newcommand{\pr}{\operatorname{pr}\nolimits}
\newcommand{\res}{\operatorname{res}\nolimits}
\newcommand{\Res}{\operatorname{Res}\nolimits}

\newcommand{\Sym}{\operatorname{Sym}\nolimits}
\newcommand{\Tr}{\operatorname{Tr}\nolimits}

\newcommand{\Ext}{\operatorname{Ext}\nolimits}

\newcommand\lra{{\longrightarrow}}

\newcommand\ra{{\rightarrow}}
\newcommand\rar{{\rightarrow}}

\newcommand{\mapor}[1]{{\stackrel{#1}{\longrightarrow}}}
\newcommand{\mapver}[1]{\Big\downarrow\vcenter{\rlap{$\scriptstyle#1$}}}

\begin{document}
%\title{KURANISHI Spaces}
\title{Kuranishi Spaces of meromorphic connections}
\author{Francois-Xavier Machu}
%\affil \endaffil
\address{F-X. M.: Department of Mathematical and Statistical Sciences,
632 CAB, University of Alberta, Edmonton, AB T6G 2G1, Canada}
% \curaddr \endcuraddr
\email{machu@ualberta.ca}
\thanks{Partially supported by grant FWF-AP19667}

\bigskip
\bigskip
\begin{abstract}
We construct the Kuranishi spaces, or in other words, the versal deformations, for the following classes of connections 
with fixed divisor of poles $D$: all such connections, as well as for its subclasses of integrable,
integrable logarithmic and integrable logarithmic connections with a parabolic structure 
over $D$ . The tangent and obstruction spaces of deformation theory are defined as the hypercohomology of an appropriate
complex of sheaves, and the Kuranishi space is a fiber of the formal obstruction map.  
\end{abstract}
\maketitle
\vspace{1 ex}
\begin{center}\begin{minipage}{110mm}\footnotesize{\bf Key words:} connections, deformations, hypercohomology, Kuranishi spaces
obstructions.
\end{minipage}
\end{center}

\begin{center}\begin{minipage}{110mm}\footnotesize{\bf MSC2000:} 14B12, 14F05, 14F40, 14H60, 32G08.

\end{minipage}
\end{center}
\vspace{1 ex}

\section*{Introduction}
We construct the Kuranishi space, or in other words, the versal deformation, of connections
belonging to each one of the following classes:\\
meromorphic connections with fixed divisor of poles $D$;\\
integrable meromorphic connections with fixed divisor of poles $D$;\\
integrable logarithmic connections with fixed divisor of poles $D$;\\
integrable logarithmic connections on curves with parabolic structure at singular points.

The interest in versal deformations is twofold.
First, a versal deformation is a kind of a local moduli space which exists in a much wider
range of situations than the moduli spaces in the proper sense do. Second, versal deformations are usually easier to
write down than the moduli spaces, and one can use the versal deformation to determine the germ of the moduli space
up  to analytic, formal or \'etale equivalence.

Historically, versal deformations were introduced for the first time in late $50$'s in the work of Kodaira and Spencer (\cite{KS-1},\cite{KS-2}), and Kuranishi (\cite{Ku-1},\cite {Ku-2}).
In the beginning, this theory was only concerned with deformations of compact complex manifolds and was viewed as a replacement
for Riemann's insight of moduli of compact complex curves in higher dimensions. But since then the theory has been significantly
formalized and extended to a much wider range of domains: singularities \cite{Ar}, \cite{Schl-2}, \cite{AGZV}, 
vector bundles and sheaves  \cite{Rim-1}, \cite{Rim-2}, \cite{Artam-1}, \cite{Artam-2},
singular complex spaces \cite{Gro}, \cite{Illu-1},\cite{Illu-2}, \cite{Pa-1},\cite{Pa-2}, and morphisms of varieties or complex spaces \cite{Fl}, \cite{Bi}, \cite{Ran-1}, \cite{Ran-2}.

Recently, many people believe that a deformation theory over a field of characteristic $0$ should be taken over
by a differential graded Lie algebra (denoted DGLA). This principle deriving from researches regarding
homotopy theory, quantization, mirror symmetry, etc. (see, for instance, \cite{Kon}).
One prototype example to this principle is the deformation theory of compact complex manifold
via Maurer-Cartan equation on the vector field valued $(0,1)$ forms. This is the Newlander-Nirenberg
theorem (or rather Kuranishi's proof of the existence of the Kuranishi space). If we restrict to infinitesimal deformations, we can describe the situation as a bijection between
\bigskip

$\frac{\{\textrm{Maurer-Cartan solutions in $KS^1_{X}\otimes m_{A}$\}}}{{\textrm{gauge equivalence}}}$ $\simeq$ $\frac{\{\textrm{ deformations of $X$ on $A$\}}}{\textrm{isomorphisms}}$, where $A$ is a local artinian $\CC$-algebra and $KS^{\fatdot}_X=(A_X^{0,\fatdot}(\Theta_X),\partial,[-,-])$
the Kodaira-Spencer algebra on $X$. This isomorphism  is functorial in $A$. The left-hand side is the deformation functor
associated to the Kodaira-Spencer DGLA $KS_X^{\fatdot}$, denoted by $\Def_{KS_{X}}$, and the right-hand side is the usual
deformation functor $\Def_X$ of $X$.

All the constructions are enclosed in the paradigm of the Kuranishi space associated to a "good" deformation theory.   
A "good" deformation theory for some type of object $X$ consists in determining a triple 
$(T^1_X,T^2_X,f)$, where $T^1_X$ is the tangent space to deformations of $X$, $T^2_X$ is the obstruction space,
$f:\hat{T^1_X}\rar\hat{T^2_X}$ a formal map without linear terms, called the Kuranishi map (\ $\hat{}$ denotes the formal completion
at zero). Then the formal scheme $f^{-1}(0)$ is the Kuranishi space, or a formal germ of the versal deformation
of $X$.

We provide the triples $(T^1_X,T^2_X,f)$ for the above four classes of connections. In all the $4$ cases,
$T^i_X=\HH^i(\CCC^{\fatdot})$, the hypercohomology of an appropriate complex of sheaves, and the initial 
component $f_2$ of $f$ is the Yoneda square map. For instance, in the case $X=(\EEE,\nabla)$ is a meromorphic connection
with fixed divisor of poles $D$, the complex $\CCC^{\fatdot}$ is a two-term one and is
$$\CCC^{\fatdot}=[\mathcal{E}nd(\EEE)\xymatrix@1{\ar[r]^{\nabla}&}\mathcal{E}nd(\EEE)\otimes\Omega^1(D)].$$

A similar situation occurs in the deformation theory of Higgs bundles or Hitchin pairs
\cite{B-R}, where 
$T^1_X=\HH^1(\CCC^{\fatdot})$ with complex
\begin{equation*} \CCC^{\fatdot}=[\mathcal{E}nd(\EEE)\xymatrix@1{\ar[r]^{\ad\phi}&}\mathcal{E}nd(\EEE)\otimes\Omega^1(D)]\end{equation*} defined by the Higgs field $\phi:\EEE\rar\EEE\otimes\Omega^1(D)$;
contrary to our case, $\ad\phi$ is $\OOO_X$-linear.

Let $X$ be a complete scheme of finite type over $k$ or a compact complex space (then $k=\CC$). 
The existence of a versal deformation and the theoretical approach to its construction
are known for coherent sheaves on $X$. The construction of the Kuranishi space ($=$
versal deformation) for coherent sheaves is done in using the injective resolutions. We are studying vector bundles $\EEE$ with
an additional structure (a connection $\nabla$), and in this case the deformation theory
of both $\EEE$ and $(\EEE,\nabla)$ can be stated in terms of the \v Cech cohomology of a sufficiently fine open covering of $X$. This approach is easier than the one via injective resolutions. We start by the construction of the Kuranishi
space of vector bundles serving as a model for that of the pairs $(\EEE,\nabla)$.
This is done in Sect. \ref{Kvb}, where it is also explained how the versal deformations can be used to construct
analytic moduli spaces of simple vector bundles. In Sect.  \ref{Connections}, we introduce connections with fixed divisor of poles and show
that their isomorphism classes of first order deformations are classified by the hypercohomology $\HH^1(\CCC^{\fatdot})$ of
some two-term complex of sheaves. In Sect.  \ref{Obstruc}, we show that the first obstruction to lifting the first order deformation
is given by the Yoneda square and construct the Kuranishi space. We also define several versions of the Atiyah class. In Sect.  \ref{Int connec},
we describe the construction of the Kuranishi space for integrable and integrable logarithmic connections.
The last Sect.  \ref{Par connec} treats the Kuranishi space of parabolic connections. 

\subsection{Deformation theory}
In this section, we follow \cite{Ma}, and \cite{H-L} to remind the framework
of the deformation theory.

Let $\boldsymbol{\Art}$ be the category of local artinian $\CC$-algebra $A$ such that $A/m_A\simeq\CC$, where $m_A$
is the maximal ideal of $A$. We mean by a functor of artinian rings \cite{Schl-1} a covariant functor

$D:\boldsymbol{\Art}\rar\boldsymbol{\Set}$ such that $D(\CC)$ is the one-point set.
The tangent space $T_D$ to a functor of artinian rings $D$ is defined by
$T_D=D(\CC[\epsilon])$, where $\CC[\epsilon]$ is the ring of dual numbers $\CC[x]/(x^2)$.

Let $A$, $B$, $C$ be local artinian $\CC$-algebras and
$\eta:D(B\times_{A} C)\rar D(B)\times_{D(A)} D(C)$ be the natural map.
We call a functor of artinian rings $D$ a deformation functor if it satisfies
$(i)$ if $B\rar A$ is onto, so is $\eta$, and
$(ii)$ if $A=\CC$, $\eta$ is bijective \cite{Ma}(Definition 2.5).
Note that these conditions are closely related to Schlessinger's criterion of existence of a hull
(see Remark to Definition 2.7 in \cite{F-M}).

An obstruction theory of a functor of artinian rings $D$ is a pair $(U,\ob(-))$, consisting of a finite dimensional $\CC$-vector space $U$, the obstruction space, and a map $\ob(\alpha):D(A')\rar U\otimes a$, the obstruction map such that for any small extension $$\alpha: 0\ra a\ra A\ra A'\ra 0,$$ with kernel $a$ such that $m_Aa=0$,  the following conditions are satisfied:\\
$1$. If $x'\in D(A')$ lifts to $D(A)$, then $\ob(\alpha)(x')=0$.\\
$2$. For any morphism $\phi$ of small extensions
\[ \begin{array}{cccccccccc}
\alpha_1\colon\quad&0&\mapor{}&a_1&\mapor{}&A_1&\mapor{}&A'_1&\mapor{}&0\\
&&&\mapver{\phi_a}&&\mapver{\phi}&&\mapver{\phi'}&&\\
\alpha_2\colon\quad&0&\mapor{}&a_2&\mapor{}&A_2&\mapor{}&A'_2&\mapor{}&0,\end{array}
\]
we have the compatibility $\ob(\alpha_2)(\phi_{*}'(x'))=(\id_U\otimes\phi_a)(\ob(\alpha_1)(x')),$ for every $x'\in D(A'_1)$.
Moreover, if $\ob(\alpha)(x')=0$ implies the existence of a lifting of $x'$ to $D(A)$, the obstruction is called complete.

In the sequel, we always assume that $k$ is an algebraically closed field or $k=\CC$. 
For instance, if $X$ is a smooth projective variety
over $k$, and let $F$ be a coherent $\OOO_X$-module which is simple. If $A\in\boldsymbol{\Art/k}$, let
$D_F(A)$ be the set of isomorphim classes of pairs $(F_A,\phi)$ where $F_A$ is a flat family
of coherent sheaves on $X$ parameterized by $\Spec(A)$ and $\phi:F_A\otimes_A k\rar F$ is an isomorphism
of $\OOO_X$-modules.
Following \cite{H-L}, the map $D_F(\alpha):D_F(A)\rar D_F(A')$ has for fibers affine spaces with
affine group $\Ext^1(F,F)\otimes_k a$, and the image of $D_F(\alpha)$ lies in the kernel of
the obstruction map $\ob(\alpha):D_F(A')\rar\Ext^2(F,F)\otimes_k a$.
\begin{proposition}(See[\cite{Ma}, Proposition 2.17].)  
Let $D_1$ and $D_2$ be deformation functors and $\phi:D_1\rar D_2$ a morphism
of functors, $(V_1, \ob_{D_1})$ and $(V_2, \ob_{D_2})$ obstruction theories for $D_1$ and $D_2$,
respectively. Assume that \\
$(i)$ $\phi$ induces a surjection (resp. bijection) on the tangent spaces $T_{D_1}\rar T_{D_2}$.\\
$(ii)$ There is an injective linear map between obstruction spaces $\Phi:V_1\rar V_2$ such that $\ob_{D_2}\circ\phi=\Phi\circ\ob_{D_1}$.\\
$(iii)$ The obstruction theory $(V_1, ob_{D_1})$ is complete.\\
Then, the morphism $\phi$ is smooth (resp \'etale).
\end{proposition}

\section{Construction of the Kuranishi space in the case of vector bundles over any base.}\label{Kvb}

Let $X$ be a complete scheme of finite type over $k$ or a complex space (then $k=\CC$), $\mathfrak U=(U_{\alpha})$ be an open covering of $X$,
$e_{\alpha}$ a trivialization of $\EEE_{\rvert{U_{\alpha}}}$.
The transition functions $g_{\alpha \beta}$ relate the trivializations by the formula $e_{\beta}=e_{\alpha} g_{\alpha \beta}$ over $U_{\alpha \beta}=U_{\alpha}\cap U_{\beta}$ and satisfy the following relations
\begin{equation}\label{equation}
g_{\alpha \beta}=g^{-1}_{\beta \alpha},\ g_{\alpha \beta}g_{\beta \gamma}g_{\gamma \alpha}=1.
\end{equation}
In other words, $(g_{\alpha \beta})\in\check{C}^{1}(\mathfrak U, \GL(r, \OOO_X))$ is a skew-symmetric multiplicative $1$-cocycle.

\subsection{Construction of the Kuranishi space in the case of simple vector bundles over any base}
\begin{definition}
A vector bundle $\EEE$ on $X$ is simple if and only if $H^0(X,\mathcal{E}nd(\EEE))= k\id$.
\end{definition}
In the case of a simple vector bundle, the versal deformation is in fact universal and this is a local version of the moduli space:
\begin{proposition}\label{svb}
Let $\EEE$ be a simple vector bundle on a scheme $X$ of finite type on $k$ or a complex space (in which case $k=\CC$). Then there exists an analytic  space $M(\EEE)$ with a reference point $*$ and a vector bundle $E$ on $X\times M(\EEE)$ which satisfy the following properties:\\ 
$(1)$ $E\vert_{X\times *}\simeq \EEE$.\\
$(2)$ If $T$ is an analytic space with a reference point $*$ and $E'$ a vector bundle on $X\times T$ such that
$E'\lvert_{X\times *}\simeq\EEE$, then there is a holomorphic mapping $\Phi:T\rar M(\EEE)$ such that $\Phi(*)=*$ and
$E'\simeq(1\times \Phi)^{*}(E)$.\\ 
$(3)$ The above mapping $\Phi$ is unique as a germ of a holomorphic mapping from $(T, *)$ to $(M(\EEE), *).$ 
$(M(\EEE), *)$ and $E$ are called the Kuranishi space and the Kuranishi family of $\EEE$, respectively.
\end{proposition}
\begin{proof}
See \cite{Mu-1}.
\end{proof}
We define $SV_X$ as the set of isomorphism classes of simple vector bundles on $X$. Using Proposition \ref{svb}, we can endow it with an analytic structure so that $SV_X$ has a universal family only locally in the \'etale or classical topology. Then there exists a sufficiently small open set $U$ of $SV_X$ in the classical or \'etale topology and a vector bundle $E$ on $X\times U$ satisfiying the following property: For any analytic space $S$, there exists a functorial bijection between the sets $\{\textrm{morphisms $S\rar U$\}}$ $\rar$ $\{\textrm{vector bundles $E$ on $X\times S$ such that $\forall s\in S$, $E_s$ is simple and its class belongs to $U$\}}/\sim$ given by $\phi\mapsto(1\times\phi)^{*}(E).$
\begin{proposition}
Let $X,\EEE$ be as in Proposition \ref{svb}. 
Every obstruction to the smoothness of $SV_X$ at $[\EEE]$ lies in $\ker(H^2(\Tr):H^2(X,\mathcal{E}nd(\EEE))\rar H^2(X,\OOO_X))$.
In particular, $SV_X$ is smooth at $[\EEE]$ if $H^2(\Tr)$ is injective.
\end{proposition}
\begin{proof}
See \cite{Mu-1}.
\end{proof}

Note, however, that $SV_X$, even if it is smooth, is not a nice concept of moduli space: it is non-separated in many examples.
\\

We now treat the case of vector bundles over any base.
\subsection{First order deformations}
Deform the transition functions:
$\tilde{g}_{\alpha \beta}=g_{\alpha \beta}+\epsilon g_{\alpha \beta,1}$, where $g_{\alpha \beta,1}\in\Gamma(U_{\alpha \beta},M_r(\OOO_X))$ and $\epsilon^2=0$.
We have $g_{\alpha \beta,1}=\frac{\ud\tilde{g}_{\alpha \beta}}{\ud\epsilon}$. Differentiating (\ref{equation}), we obtain:
\begin{equation}\label{equation 1}g_{\beta \alpha,1}=\frac{\ud\tilde{g}_{\alpha \beta}^{-1}}{\ud\epsilon}=-g^{-1}_{\alpha \beta}g_{\alpha \beta,1}g^{-1}_{\alpha \beta}, \end{equation}
$$g_{\alpha \beta,1}g_{\beta \gamma}g_{\gamma \alpha}+g_{\alpha \beta}g_{\beta \gamma,1}g_{\gamma \alpha}+g_{\alpha \beta}g_{\beta \gamma}g_{\gamma \alpha,1}=0,$$
and by (\ref{equation 1}), $g_{\gamma \alpha,1}=-g^{-1}_{\alpha \gamma}g_{\alpha \gamma,1}g^{-1}_{\alpha \gamma}.$ 
Plugging this into the previous formula, we get 
$$g_{\alpha \beta,1}g_{\beta \gamma}g_{\gamma \alpha}+g_{\alpha \beta}g_{\beta \gamma,1}g_{\gamma \alpha}=g_{\alpha \beta}g_{\beta \gamma}g^{-1}_{\alpha \gamma}g_{\alpha \gamma,1}g^{-1}_{\alpha \gamma}.$$
Multiply by $g_{\alpha \gamma}$ on the right:
\begin{equation}\label{equation 2}
g_{\alpha \beta,1}g_{\beta \gamma}+g_{\alpha \beta}g_{\beta \gamma,1}=g_{\alpha \gamma,1}.
\end{equation}
We want to represent this in the form $a_{\alpha \beta}+a_{\beta \gamma}=a_{\alpha \gamma}$ for an appropriate additive $1$-cocycle
$a=(a_{\alpha \beta})\in\check{C}^{1}(\mathfrak U,\mathcal{E}nd(\EEE))$, associated with $(g_{\alpha \beta,1})$ and skew-symmetric:
$a_{\alpha \beta}=-a_{\beta \alpha}.$
Define $a_{\alpha \beta}\in\Gamma(U_{\alpha \beta}, \mathcal{E}nd(\EEE))$ by its matrix: $g^{-1}_{\alpha \beta}g_{\alpha \beta,1}$
in the basis $e_{\beta}$ and $g_{\alpha \beta,1}g^{-1}_{\alpha \beta}$ in the basis $e_{\alpha}$. Then (\ref{equation 1}) gives
$g_{\alpha \beta}g_{\beta \alpha,1}+g_{\alpha \beta,1}g^{-1}_{\alpha \beta}=0$, written in terms of matrices with
respect to the basis $e_{\alpha}$, and (\ref{equation 2}) amounts to  $a_{\alpha \beta}+a_{\beta \gamma}=a_{\alpha \gamma}.$
Thus the first order deformations of $\EEE$ are classified by the $1$-cocycles $a=(a_{\alpha \beta})\in\check{C}^{1}(\mathfrak U,\mathcal{E}nd(\EEE))$. Such a deformation is trivial if the vector bundle $\tilde{\EEE}$ defined over $X\times\Spec\CC[\epsilon]/(\epsilon^2)$ by the $1$-cocycle
$\tilde{g}_{\alpha \beta}=g_{\alpha \beta}+\epsilon g_{\alpha \beta,1}$ is isomorphic to $\pr_{1}^{*}(\EEE)$, where
$\pr_1:X\times\Spec\CC[\epsilon]/(\epsilon^2)\rar X$ is the natural projection.
This means that there exists a change of basis 
$e_{\alpha}\mapsto \tilde{e}_{\alpha}=e_{\alpha}(1+\epsilon h_{\alpha})$ which transforms $\tilde{g}_{\alpha \beta}$ into $g_{\alpha \beta}$.
We compute $\tilde{e}_{\beta}=e_{\beta}(1+\epsilon h_{\beta})=e_{\alpha} g_{\alpha \beta}(1+\epsilon h_{\beta})=\tilde{e}_{\alpha}(1-\epsilon h_{\alpha})g_{\alpha \beta}(1+\epsilon h_{\beta})$ and we want that this coincides with  $\tilde{e}_{\beta}=\tilde{e}_{\alpha}\tilde{g}_{\alpha \beta}.$
That is: $g_{\alpha \beta}+\epsilon g_{\alpha \beta,1}=(1-\epsilon h_{\alpha})g_{\alpha \beta}(1+\epsilon h_{\beta})$, or
$g_{\alpha \beta,1}=-h_{\alpha}g_{\alpha \beta}+g_{\alpha \beta}h_{\beta}$. Interpreting $h_{\alpha}$ as the matrix of $b_{\alpha}\in\Gamma(U_{\alpha}, \mathcal{E}nd(\EEE))$ with respect to the basis $e_{\alpha}$, we obtain $a_{\alpha, \beta}=-b_{\alpha}+b_{\beta}$ which
is written in the basis $e_{\alpha}$ in the form $g_{\alpha \beta,1}g^{-1}_{\alpha \beta}=-h_{\alpha}+g_{\alpha \beta}h_{\beta}g^{-1}_{\alpha \beta}.$
Thus the equivalence classes of first order deformations of $\EEE$ over $V=\Spec\CC[\epsilon]/(\epsilon^2)$ are classified by 
$$\check{H}^{1}(\mathfrak U, \mathcal{E}nd(\EEE))=\frac{\{\textrm{$1$-cocycles $(a_{\alpha \beta})\in \check{C}^{1}(\mathfrak U,\mathcal{E}nd(\EEE))$}\}}{\{\textrm{coboundaries $a_{\alpha \beta}=b_{\beta}-b_{\alpha}$, where $(b_{\alpha})\in\check{C}^{0}(\mathfrak U,\mathcal{E}nd(\EEE))$}\}}.$$
\subsection{First obstruction}
We denote $V_k=\Spec \CC[\epsilon]/(\epsilon)^{k+1}$. We will investigate the following question: which of the deformations of $\EEE$ over $V_1$ lift to $V_2$?\\
Let $G_{\alpha \beta}=g_{\alpha \beta,0}+\epsilon g_{\alpha \beta,1}+\epsilon^2 g_{\alpha \beta,2}$ be a deformation of the
cocycle $g_{\alpha \beta}=g_{\alpha \beta,0}$ over $V_2$.\\
We want to prove, in other words that $G_{\alpha \beta}$ gives a valid $2$nd-order deformation if and only if it satisfies
the cocycle condition.

Assume that $G_{\alpha \beta}\mod\epsilon^2$ is a $1$-cocycle, then (\ref{equation 1}) and (\ref{equation 2}) are verified, and compute the coefficient $K_{\alpha\beta\gamma,2}$ of $\epsilon^2$ in $G_{\alpha \beta}G_{\beta \gamma}G_{\gamma \alpha}$, which will be denoted $K_{\alpha \beta \gamma,2}$:\begin{eqnarray}\label{eqnarray g}
K_{\alpha \beta \gamma,2}=g_{\alpha \beta,0}g_{\beta \gamma,1}g_{\gamma \alpha,1}+g_{\alpha \beta,1}g_{\beta \gamma,0}g_{\gamma \alpha,1}+g_{\alpha \beta,1}g_{\beta \gamma,1}g_{\gamma \alpha,0} \\ \nonumber+g_{\alpha \beta,2}g_{\beta \gamma,0}g_{\gamma \alpha,0}+g_{\alpha \beta,0}g_{\beta \gamma,2}g_{\gamma \alpha,0}+g_{\alpha \beta,0}g_{\beta \gamma,0}g_{\gamma \alpha,2}\end{eqnarray}
Similar to the above, introduce the sections $a_{\alpha \beta,i}$, $(i=1, 2)$ of the endomorphism sheaf $\mathcal{E}nd(\EEE_{\rvert(U_{\alpha \beta})})$ having $g_{\alpha \beta,i}g^{-1}_{\alpha \beta}$ for their matrices in the bases $e_{\alpha}$. Then, as above, $g_{\alpha \beta,2}g_{\beta \gamma,0}g_{\gamma \alpha,0}+g_{\alpha \beta,0}g_{\beta \gamma,2}g_{\gamma \alpha,0}+g_{\alpha \beta,0}g_{\beta \gamma,0}g_{\gamma \alpha,2}$ is the matrix of $a_{\alpha \beta,2}+a_{\beta \gamma,2}+a_{\gamma \alpha,2}$ in the basis $e_{\alpha}$, and
$g_{\alpha \beta,0}g_{\beta \gamma,1}g_{\gamma \alpha,1}+g_{\alpha \beta,1}g_{\beta \gamma,0}g_{\gamma \alpha,1}+g_{\alpha \beta,1}g_{\beta \gamma,1}g_{\gamma \alpha,0}$ is the matrix of \begin{equation}\label{equation 3}
a_{\beta \gamma,1}a_{\gamma \alpha,1}+a_{\alpha \beta,1}a_{\gamma \alpha,1}+a_{\alpha \beta,1}a_{\beta \gamma,1} \end{equation} in the basis $e_{\alpha}$.
Let $a_1$ denote the cocycle $(a_{\alpha \beta,1})$ and $[a_1]$ its class in $\check{H}^{1}(\mathfrak U, \mathcal{E}nd(\EEE))$. Then 
$a_{\beta \gamma,1}a_{\gamma \alpha,1}=c_{\beta \gamma \alpha}$ represents the Yoneda product $[a_1]\circ[a_1]=[c]\in \check{H}^{2}(\mathfrak U, \mathcal{E}nd(\EEE))$; see for instance $10.1.1.$ of \cite{H-L} for the definition of the Yoneda product
$$\check{H}^{i}(\mathfrak U, \mathcal{E}nd(\EEE))\times\check{H}^{j}(\mathfrak U, \mathcal{E}nd(\EEE))\rar\check{H}^{i+j}(\mathfrak U, \mathcal{E}nd(\EEE)).$$
The whole expression (\ref{equation 3}) is the skew-symmetrization $\hat{c}_{\alpha \beta \gamma}$ of $c_{\beta \gamma \alpha}$, hence
it represents the same cohomology class $[c]$. Let also $a_2$ denote the \v Cech cochain $(a_{\alpha \beta,2})$. 
We can rewrite $K_{2}=(K_{\alpha \beta \gamma,2})$ in the form \begin{equation}\label{equation a}K_{2}=\hat{c}+\check{d}a_2.\end{equation}
We now see that we can find $a_2$ in such a way that $(G_{\alpha \beta})$ is a cocycle over $V_2$ if and
only if $\hat{c}$ is $\check{d}$-exact. We have proved:
\begin{proposition}\label{proposition 1}
Let $X$ be a complete scheme of finite type over $k$ or a complex space (and then $k=\CC$), $\EEE$ a vector bundle on $X$, $[a]\in H^{1}(X, \mathcal{E}nd(\EEE))$. Then the first order deformation of $\EEE$ over $V_1$ defined by $[a]$ lifts to a deformation over $V_2$ if and only if the Yoneda square $[a]\circ [a]$ is zero in $H^{2}(X, \mathcal{E}nd(\EEE))$.
\end{proposition} 
\begin{definition}
The map 
\begin{eqnarray}\label{eqnarray.0..}%
H^{1}(X, \mathcal{E}nd(\EEE)) & \  \ra &  
H^{2}(X, \mathcal{E}nd(\EEE))\\ \nonumber
([a]) &  \mapsto\ & [a]\circ[a]
\end{eqnarray}    
will be called first obstruction, and denoted $ob^{(2)}$.\end{definition} 
Thus $ob^{(2)}$ is the map of taking the Yoneda square. We will now construct a universal first order deformation of $\EEE$ on $X$.
Let $W=H^{1}(X, \mathcal{E}nd(\EEE))$, $t_1,\dots,t_N$ a coordinate system on $W$, $W_k=\Spec k[t_1,\dots,t_N]/(t_1,\dots,t_N)^{k+1}$
the $k$-th infinitesimal neighborhood of the origin in $W$. The universal first order deformation $\EEE_1$ of $\EEE$ over $W_1$ can be
described as follows.

Choose an open covering of $X$ as above, so that $\EEE$ is defined by a $1$-cocyle $(g_{\alpha \beta})$. We deform
$\EEE$ by specifying a family $G_{\alpha \beta}(t_1,\dots,t_N)$ of $1$-cocyles over $X\times W_1$. Pick up $N$ cocycles $a_i=(a^{(i)}_{\alpha\beta})\in
\check{C}^1(\mathfrak{U},\mathcal{E}nd(\EEE))$ whose cohomology classes $[a_1],\dots,[a_N]$ form a basis of $W$ dual to the coordinates $t_1,\dots,t_N$.
Then we set $g^{(i)}_{\alpha\beta}=a^{(i)}_{\alpha\beta}g_{\alpha\beta}$, where $a^{(i)}_{\alpha\beta}$ is represented by its matrix in the basis $e_{\alpha}$ and write $G_{\alpha \beta}(t_1,\dots,t_N)=g_{\alpha \beta}+\sum_{i=1}^{N} g^{(i)}_{\alpha \beta}t_i$. 
Then $G_{\alpha \beta}$ is a $1$-cocycle and defines a vector bundle $\EEE_1$ over $X\times W_1$ called a universal first order deformation of $\EEE$. The whole universal deformation over $W_1$ cannot be lifted to a deformation on $W_2$. Proposition \ref{proposition 1} implies:
\begin{proposition}\label{proposition 2}
There is a maximal subscheme $K_2\subset W_2$ with the property that $\EEE_1$ extends as a vector bundle from $X\times W_1$ to $X\times K_2$.
This maximal subscheme $K_2$ is the (second infinitesimal neighborhood of the origin in the cone) defined by the equation $ob^{(2)}(z)=0$ in $W_2$.
\end{proposition}
We will now prove the following theorem, providing a construction of the formal Kuranishi space:
\begin{theorem}\label{theorem 1}
Let $X,\EEE$ be as above, $W=H^{1}(X, \mathcal{E}nd(\EEE))$, $(\delta_1,\dots,\delta_N)$ a basis of $W$ and $(t_1\dots,t_N)$ the dual coordinates on $W$. Let $W_k=\Spec k[t_1,\dots,t_N]/(t_1,\dots,t_N)^{k+1}$ be the $k$-th infinitesimal neighborhood of the origin in $W$, $\EEE_1$ a universal first order deformation of $\EEE$ over $X\times W_1$ as above. Then there exists a formal power series 
$$f(t_1,\dots,t_N)=\sum_{k=2}^{\infty} f_{k}(t_1\dots,t_N)\in H^{2}(X, \mathcal{E}nd(\EEE))[[t_1,\dots,t_N]],$$
where $f_k$ is  homogeneous of degree $k$, with the following property.
Let $I$ be the ideal of $k[[t_1,\dots,t_N]]$ generated by the image of the map $f^*:H^{2}(X, \mathcal{E}nd(\EEE))^*\rar k[[t_1,\dots,t_N]]$, adjoint to $f$.
Then for any $k\geq 2$, the universal first deformation $\EEE_1$ of $\EEE$ over $X\times W_1$ extends to a vector bundle $\EEE_k$ on $X\times K_k$, where $K_k$ is a closed subscheme of $W_k$ defined by the ideal $I\otimes k[[t_1,\dots,t_N]]/(t_1,\dots,t_N)^{k+1}.$
\end{theorem}
\begin{definition}
The inverse limit $\mathbb{K}=\underleftarrow{\lim}K_k$ is called the formal Kuranishi space of $\EEE$, and $\boldsymbol{\EEE}=\underleftarrow{\lim} \EEE_k$ the formal universal bundle over $\mathbb{K}$.
\end{definition}
\begin{proof}
Let $\mathfrak U=(U_k)$ be an open covering, sufficiently fine so that $\EEE_{\vert_{U_{\alpha}}}$ is trivialized
by a basis $e_{\alpha}$, and the groups $H^{i}(X,\End(\EEE))$ are computed by the $\check{C}$ech complex $(\check{C}^{\fatdot}(\mathfrak U, \mathcal{E}nd(\EEE)),\check{d})$. Let $\check{Z}^{i}(\mathfrak U, \mathcal{E}nd(\EEE)), \check{B}^{i}(\mathfrak U, \mathcal{E}nd(\EEE))$ denote the subspaces of cocycles and coboundaries in $\check{C}^{i}(\mathfrak U, \mathcal{E}nd(\EEE))$ respectively. 
Let us fix some cross-sections $\sigma_i:H^{i}(X,\mathcal{E}nd(\EEE))\rar\check{Z}^{i}(\mathfrak U,\mathcal{E}nd(\EEE))$ and $\tau:\check{B}^{2}(\mathfrak U, \mathcal{E}nd(\EEE))\rar\check{C}^{1}(\mathfrak U, \mathcal{E}nd(\EEE))$ of the natural maps in the opposite
direction. Let $a_i=(a^{(i)}_{\alpha \beta})=\sigma_1(\delta_i)$, and denote, as above, by $(g_{\alpha \beta})$ the $1$-cocycle defining
$\EEE$, so that $e_{\beta}=e_{\alpha}g_{\alpha \beta}$. We will construct by induction on $k\geq 0$ the homogeneous forms of degree $k$ in $t_1,\dots,t_N$
\begin{eqnarray}\label{eqnarray.1..}%
G_{\alpha \beta, k}(t_1,\dots,t_N)\in\Gamma(U_{\alpha \beta},M_r(\OOO_X))\otimes k[t_1,\dots,t_N], & \\ \nonumber
F_{\alpha \beta \gamma, k}(t_1,\dots,t_N)\in\Gamma(U_{\alpha \beta \gamma},\mathcal{E}nd(\EEE))\otimes k[t_1,\dots,t_N], & \\ \nonumber
f_k(t_1,\dots,t_N)\in H^{2}(X,\mathcal{E}nd(\EEE))\otimes k[t_1,\dots,t_N]\end{eqnarray}
with the following properties:\\ 
$(i)$ $G_{\alpha\beta,0}=g_{\alpha\beta},G_{\alpha\beta,1}=\sum_{i=1}^{N}a^{(i)}_{\alpha\beta}g_{\alpha\beta}t_i$, where $a^{(i)}_{\alpha\beta}$ are represented by their matrices in the basis $e_{\alpha}$.\\
$(ii)$ $f_k=0, F_{\alpha \beta \gamma, k}=0$ for $k=0,1$.\\
$(iii)$ For each k $\geq 1$, let $f^{(k)}=\sum_{i\leq k}f_i$, and let $I^{(k+1)}$ be the ideal generated by $(t_1,\dots,t_N)^{k+2}$ and the image of the adjoint map $f^{(k)*}:H^{2}(X, \mathcal{E}nd(\EEE))^*\rar k[t_1,\dots,t_N]$. Then $(F_{\alpha \beta \gamma, k+1})$ is
a cocycle modulo $I^{(k+1)}$ and $f_{k+1}$ is a lift to $H^{2}(X, \mathcal{E}nd(\EEE))\otimes k[t_1,\dots,t_N]$ of the cohomology class
$[(F_{\alpha \beta \gamma, k+1}\mod I^{(k+1)})]\in H^{2}(X, \mathcal{E}nd(\EEE))\otimes k[t_1,\dots,t_N]/I^{(k+1)}$.\\
$(iv)$ For any $k\geq 1$, set $G^{(k)}_{\alpha \beta}=\sum_{i\leq k}G_{\alpha \beta, i}$. Then $G^{(k)}_{\alpha \beta}G^{(k)}_{\beta \gamma}G^{(k)}_{\gamma \alpha}\equiv (1+F_{\alpha \beta \gamma, k+1})\mod I^{(k+1)}.$
Properties $(i), (ii)$ determine $G_{\alpha \beta, k}, F_{\alpha \beta \gamma, k}$ for $k\leq 1$.

The proof of Proposition \ref{proposition 1}
allows us to see that $(iii), (iv)$ are verified for $k=1$ with
$$F_{\alpha \beta \gamma, 2}=\sum_{i,j=1}^{N}(a^{(i)}_{\beta \gamma}a^{(j)}_{\gamma \alpha}+a^{(i)}_{\alpha \beta}a^{(j)}_{\gamma \alpha}+
a^{(i)}_{\alpha \beta}a^{(j)}_{\beta \gamma})t_it_j.$$ and to determine $G_{\alpha \beta, 2}$ we proceed as follows.
Let $f_2=[(F_{\alpha\beta \gamma,2})]$, and $I^{(2)}$ be the ideal of $K_2$, that is the ideal generated by $(t_1,\dots,t_N)^3$ and the image of 
the adjoint map $f^{(2)*}:H^{2}(X, \End(\EEE))^*\rar k_2[t_1,\dots,t_N]=\Sym^{2}(W^{*})$ (the degree-$2$ homogeneous part of $k[t_1,\dots,t_N]$).
Then the reduction $\mod I^{(2)}$ of $F_2=(F_{\alpha\beta \gamma,2})$ is an element $\bar{F_2}=(F_{\alpha\beta \gamma,2})\mod I^{(2)}\in\check{B}^{2}(\mathfrak U, \mathcal{E}nd(\EEE))\otimes (\Sym^{2}(W^{*})/I^{(2)}\cap\Sym^{2}(W^{*})).$ We define a skew-symmetric $1$-cochain $a_2=a_{\alpha \beta,2}\in\check{C}^{1}(\mathfrak U,\mathcal{E}nd(\EEE))\otimes\Sym^{2}(W^{*})$ as an arbitrary lift of $(\tau\otimes\id)(\bar{F_2})\in\check{C}^{1}(\mathfrak U,\mathcal{E}nd(\EEE))\otimes(\Sym^{2}(W^{*})/I^{(2)}\cap\Sym^{2}(W^{*}))$ under the quotient map. Next we define $G_{\alpha \beta,2}$ by $G_{\alpha \beta,2}=a_{\alpha \beta,2}g_{\alpha \beta}$, where the matrix of $a_{\alpha \beta,2}$ is taken in the basis $e_{\alpha}$.

Likewise, assuming that $G^{(k-1)}_{\alpha \beta}$, $F^{(k)}_{\alpha \beta}$ are already fixed, we
can choose $F_{\alpha \beta \gamma, k+1}$ and $G_{\alpha \beta,k}$ as follows. By the induction hypothesis, we
have $G^{(k-1)}_{\alpha \beta}G^{(k-1)}_{\beta \gamma}G^{(k-1)}_{\gamma \alpha}\equiv(1+F_{\alpha \beta \gamma,k})\mod I^{(k)}.$
Then $(F_{\alpha \beta \gamma, k})$ is a cocycle modulo $I^{(k)}$, and is a coboundary modulo $I^{(k+1)}$:
$\bar{F_k}=(F_{\alpha \beta \gamma, k}\mod I^{(k+1)})\in\check{B}^{2}(\mathfrak U, \mathcal{E}nd(\EEE))\otimes (\Sym^{k}(W^{*})/I^{(k+1)}\cap\Sym^{k}(W^{*})).$
We define $G_{\alpha \beta,k}=a_{\alpha \beta,k}g_{\alpha \beta}$ with $(a_{\alpha \beta,k})\in\check{C}^{1}(\mathfrak U,\End(\EEE))\otimes\Sym^{k}(W^{*})$ an arbitrary skew-symmetric lift to $\Sym^{k}(W^{*})$ of $(\tau\otimes\id)(\bar{F_k}).$
Then $G^{(k)}_{\alpha \beta}G^{(k)}_{\beta \gamma}G^{(k)}_{\gamma \alpha}\equiv 1\mod (I^{(k+1)}+(t_1,\dots,t_N)^{(k+1)}),$
and we can define $F_{\alpha \beta \gamma, k+1}$ as the degree-$(k+1)$ homogeneous component of $G^{(k)}_{\alpha \beta}G^{(k)}_{\beta \gamma}G^{(k)}_{\gamma \alpha}.$
To end this inductive construction of the sequences $G_{\alpha \beta,k}$, $F_{\alpha \beta \gamma, k+1}$, we need only to
prove that $F_{k+1}=(F_{\alpha \beta \gamma,k+1})$ is a $2$-cocycle modulo $I^{(k+1)}$ with values in $\mathcal{E}nd(\EEE)$.
\end{proof}
The latter is proved in Lemma \ref{lemma 1} below.
\begin{lemma}\label{lemma 1}
The $2$-cochain $(F_{\alpha \beta \gamma, k+1})$, constructed in the proof of Theorem \ref{theorem 1} as the
degree-$(k+1)$ homogeneous component of $G^{(k)}_{\alpha \beta}G^{(k)}_{\beta \gamma}G^{(k)}_{\gamma \alpha}$, is
a $2$-cocycle modulo $I^{(k+1)}$ with values in $\mathcal{E}nd(\EEE)$.
\end{lemma}
\begin{proof}
The hypotheses, under which we have to prove the assertion of lemma \ref{lemma 1}, are the following:
$G^{(k)}_{\alpha \beta}=\sum_{i=0}^{k} G_{\alpha \beta,i}\in\Gamma(U_{\alpha \beta}, M_r(\OOO_X))\otimes k[t_1,\dots,t_N]$ 
are the matrix polynomials of degree $\leq k$ in $t_1,\dots,t_N$ and there is an ideal $J\subset(t_1,\dots,t_N)^{2}$
such that $G^{(k)}_{\alpha \beta}G^{(k)}_{\beta \alpha}\equiv 1\mod J$ and $G^{(k)}_{\alpha \beta}G^{(k)}_{\beta \gamma}G^{(k)}_{\gamma \alpha}
\equiv 1\mod(J+(t_1,\dots,t_N)^{k+1}).$
The ideal $J$ in Theorem \ref{theorem 1} is $I^{(k+1)}$. The collection $(F_{\alpha \beta \gamma, k})$ is considered not as a $2$-cochain
in $M_r(\OOO_X)$, but as a $2$-cochain in $\mathcal{E}nd(\EEE)$, $\EEE$ being defined by the multiplicative cocycle $(g_{\alpha \beta})=G_{\alpha \beta,0}\in\check{Z}^{1}(\mathfrak U, \GL_r(\OOO_X))$. Thus $F_{\alpha \beta \gamma}=F_{\alpha \beta \gamma, k+1}$ is a certain section of $\mathcal{E}nd(\EEE)$ over $U_{\alpha \beta \gamma}$ given by its matrix in the basis $e_{\alpha}$ of $\EEE_{\vert U_{\alpha \beta \gamma}}$.
We want to show that 
\begin{equation}\label{equation 4} F_{\alpha \beta \gamma}-F_{\alpha \beta \delta}+F_{\alpha \gamma \delta}-F_{\beta \gamma \delta}\equiv 0\mod J
\end{equation}
We will replace it by a slightly different identity \begin{equation}\label{equation 5} F_{\alpha \beta \gamma}+F_{\alpha \gamma \delta}+F_{\alpha  \delta \beta}+F_{\beta  \delta \gamma}\equiv 0\mod J,
\end{equation}
which is the same as (\ref{equation 4}) as soon as we know that $(F_{\alpha \beta \gamma})$ is skew symmetric. We have:
\begin{multline*} F_{\alpha \beta \gamma}=[G_{\alpha \beta}G_{\beta \gamma}G_{\gamma \alpha}]_{k+1},
F_{\alpha \gamma \delta}=[G_{\alpha \gamma}G_{\gamma \delta}G_{\delta \alpha}]_{k+1},
F_{\alpha \delta \beta}=[G_{\alpha \delta}G_{\delta \beta}G_{\beta \alpha}]_{k+1},\\
F_{\beta \delta \gamma}=G_{\alpha \beta,0}([G_{\beta \delta}G_{\delta \gamma}G_{\gamma \beta}]_{k+1})G^{-1}_{\alpha \beta,0}=[G_{\alpha \beta} G_{\beta \delta}G_{\delta \gamma}G_{\gamma \beta}G_{\beta \alpha}]_{k+1},\end{multline*}
where we omitted the superscript $k$ in $G^{(k)}_{\alpha \beta}$, $[\dots]_{k+1}$ stands for the homogeneous component 
of degree $k+1$ in $t_1,\dots,t_N$, and all the four terms are given by their matrices in the basis $e_{\alpha}$.
Now \begin{align*} F_{\alpha \beta \gamma}+F_{\alpha \gamma \delta}+F_{\alpha  \delta \beta}+F_{\beta  \delta \gamma}=[G_{\alpha \beta}G_{\beta \gamma}G_{\gamma \alpha}+G_{\alpha \gamma}G_{\gamma \delta}G_{\delta \alpha}+G_{\alpha \delta}G_{\delta \beta}G_{\beta \alpha}+&\\ \nonumber G_{\alpha \beta} G_{\beta \delta}G_{\delta \gamma}G_{\gamma \beta}G_{\beta \alpha}]_{k+1}\equiv[G_{\alpha \beta}G_{\beta \gamma}G_{\gamma \alpha}\times G_{\alpha \gamma}G_{\gamma \delta}G_{\delta \alpha}\times G_{\alpha \delta}G_{\delta \beta}G_{\beta \alpha}&\\ \nonumber\times G_{\alpha \beta} G_{\beta \delta}G_{\delta \gamma}G_{\gamma \beta}G_{\beta \alpha}]_{k+1}\equiv 0\mod J.\end{align*}
The skew symmetry of $(F_{\alpha \beta \gamma})$ is a particular case of (\ref{equation 5}) when $\delta=\gamma$.
\end{proof}

\section{Connections}\label{Connections}
Let $X,\EEE$ be as above. A rational (or meromorphic in the case when $X$ is a complex space) connection on $\EEE$
is a $k$-linear morphism of sheaves $\nabla:\EEE\rar\EEE\otimes\Omega^{1}_X (D)$ satisfying the Leibniz rule:
$$\forall p \in X,  \forall f\in\OOO_p,  \forall s\in\EEE_p, \nabla(fs)=f\nabla s+s\otimes\ud f.$$
We assume that $D$ is an effective Cartier divisor and call $D$ the divisor of poles of $\nabla$.
We can extend $\nabla$ in a natural way to \begin{displaymath}\EEE\otimes\Omega^{\bullet}(*D)=\raisebox{-1.em}{$\stackrel{\displaystyle\underrightarrow\lim}{\scriptstyle n}\:$}\oplus_{i\geq 0}\EEE\otimes\Omega^{i}(nD)\end{displaymath}
as a $k$-linear map $\nabla:\EEE\otimes\Omega^{i}(*D)\rar\EEE\otimes\Omega^{i+1}(*D)$ satisfying the Leibniz rule
$\nabla(s\otimes\omega)=\nabla s\wedge\omega+s\otimes\ud\omega.$
The connection is integrable if $\nabla^2=0.$ In this case, $\nabla$ defines the generalized de Rham complex
\begin{equation} 0\ra \EEE(*D)\xymatrix@1{\ar[r]^{\nabla}&}\EEE\otimes\Omega^{1}(*D)\xymatrix@1{\ar[r]^{\nabla}&} \EEE\otimes\Omega^{2}(*D) \xymatrix@1{\ar[r]^{\nabla}&}\dots,\end{equation}
If $X$ is smooth at all the points of $X\setminus{D}$, then this complex is exact over $X\setminus{D}$ in all degrees different from $0$ by
the Poincar\'e lemma. Under the same assumption, the subsheaf $\EEE^h$ of sections $s$ of $\EEE\vert_{X\setminus{D}}$
satisfying $\nabla(s)=0$ is a local system of rank $r$, that is a vector bundle with constant transition functions, and
$\EEE\vert_{X\setminus{D}}=\EEE^h\otimes\OOO_{X\setminus{D}}$; the sections of $\EEE^h$ are called horizontal
sections of $(\EEE, \nabla)$. The complex defined above, when restricted to $X\setminus{D}$, is a resolution of $\EEE^h$. 

A connection $\nabla$ on $\EEE$ induces natural connections on $\EEE^{*}, \mathcal{E}nd(\EEE), (\EEE^*)^{\otimes m}\otimes\EEE^{\otimes n}$,
and more generally, on any Schur functor of $\EEE$ or $\EEE^*$. We will use in the sequel the induced connection $\nabla_{\mathcal{E}nd (\EEE)}$ on
$\mathcal{E}nd(\EEE)$. Taking a local section $\phi$ of $\mathcal{E}nd(\EEE)$, we can think of $\phi$ as a sheaf homomorphism $\EEE\rar\EEE$ over an open set $U\subset X$
, and $\nabla_{\mathcal{E}nd (\EEE)}$ is defined by 
\begin{eqnarray}\label{eqnarray.2..}%
\nonumber\nabla_{\mathcal{E}nd(\EEE)} (\phi)=\nabla\circ\phi-\phi\circ\nabla & \\ \nonumber
\nabla_{\mathcal{E}nd(\EEE)}:\mathcal{E}nd(\EEE)\rar\mathcal{E}nd(\EEE)\otimes\Omega^{1}(D)  & \\ \nonumber
\end{eqnarray}
If $\nabla$ is integrable, then $\nabla_{\mathcal{E}nd(\EEE)}$ is also integrable, and $\mathcal{E}nd(\EEE)^h=\mathcal{E}nd(\EEE^h)$.

Let now $\mathfrak U=(U_{\alpha})$ be a sufficiently fine open covering of $X$, $e_{\alpha}$ a trivialization of $\EEE$ 
over $U_{\alpha}$, $(g_{\alpha \beta})$ the transition functions of $\EEE$ with respect to the trivilizations $(e_{\alpha})$.
The connection matrices $A_{\alpha}\in\Gamma(U_{\alpha}, M_r(\OOO_X)\otimes\Omega^{1}(D))$ of $\nabla$ are defined by
$\nabla(e_{\alpha})=e_{\alpha}A_{\alpha}$. The transition rule for the matrices $A_{\alpha}$ is
\begin{equation}\label{equation 6}
A_{\beta}=g^{-1}_{\alpha \beta}\ud g_{\alpha \beta}+ g^{-1}_{\alpha \beta}A_{\alpha}g_{\alpha \beta}
\end{equation} over $U_{\alpha \beta}$.
This equation can be given a cohomological interpretation. To this end, introduce the cochains $\AAA=(\AAA_{\alpha})\in\Check{C}^{0}(\mathfrak U, \mathcal{E}nd(\EEE)\otimes\Omega^{1}(D))$, $\GGG=(\GGG_{\alpha \beta})\in\Check{C}^{1}(\mathfrak U, \mathcal{E}nd(\EEE)\otimes\Omega^{1})$ by saying that the matrix of $\AAA_{\alpha}$ (resp. $\GGG_{\alpha \beta}$) in the basis $e_{\alpha}$ is $A_{\alpha}$ (resp. $\ud g_{\alpha \beta}g^{-1}_{\alpha \beta})$.
Then $\GGG$ is a cocycle.
\begin{definition}
The cohomology class $[\GGG]$ of $\GGG$ in $H^{1}(X, \mathcal{E}nd(\EEE)\otimes\Omega^{1})$ does not depend on the
choice of trivializations $(e_{\alpha})$ and is called the Atiyah class of $\EEE$.
We will denote this class by $\At(\EEE)$ and its image in $H^{1}(X, \mathcal{E}nd(\EEE)\otimes\Omega^{1}(D))$, in $H^{1}(X, \mathcal{E}nd(\EEE)\otimes\Omega^{1}(*D))$ by $\At^D(\EEE)$,(resp. $\At^{*D}(\EEE)$).
\end{definition} 
Now we can write (\ref{equation 6}) in the form 
$$\GGG=\Check{d}\AAA,$$ and we get the following assertion: 
\begin{proposition}
Let $X, \EEE$ be as above, $D$ an effective Cartier divisor in $X$. Then $\EEE$ 
admits a connection with divisor of poles $D$ if and only if $\At^{D}(\EEE)$ vanishes in $H^{1}(X, \mathcal{E}nd(\EEE)\otimes\Omega^{1}(D))$.
\end{proposition}
Informally speaking, this property is expressed by saying that the Atiyah class is the obstruction to the existence of 
a connection on a vector bundle.
For future use, we also provide the integrability condition of $\nabla$ in terms of the local data $A_{\alpha}$:
\begin{equation}\label{equation 7}
\ud A_{\alpha}+A_{\alpha}\wedge A_{\alpha}=0
\end{equation} 
\subsection{First order deformations of connections with fixed divisor of poles $D$}
Let $(\EEE, \nabla)$ be defined as above and $V_1=\Spec k[\epsilon]/(\epsilon^{2})$.
We represent the deformed pair $(\tilde{\EEE}, \tilde{\nabla})$ over $V_1$ by the local data
$$\tilde{g}_{\alpha \beta}=g_{\alpha \beta}+\epsilon g_{\alpha \beta,1},\ 
\tilde{A_{\alpha}}=A_{\alpha}+\epsilon A_{\alpha,1}$$ 
We have already studied the compatibility conditions which guarantee that $\tilde{g}_{\alpha \beta}$
ia a cocycle; they can be stated by saying that the cochain $a=(a_{\alpha \beta})\in\check{C}^{1}(\mathfrak U,\mathcal{E}nd(\EEE))$,
defined over $U_{\alpha \beta}$ by the matrix $g_{\alpha \beta,1}g^{-1}_{\alpha \beta}$ in the basis $e_{\alpha}$, is a cocycle.
Now, we fix this cocycle and search for a cochain $(\AAA_{\alpha,1})$ compatible with $a$. 
Expanding (\ref{equation 6}) to order $1$, we obtain:
\begin{equation}\label{equation 8}
A_{\beta,1}=g_{\beta \alpha,1}\ud g_{\alpha \beta}+g_{\beta \alpha}\ud g_{\alpha \beta,1}
+g_{\beta \alpha,1}A_{\alpha}g_{\alpha \beta}+ g_{\beta \alpha}A_{\alpha,1}g_{\alpha \beta}+
g_{\beta \alpha}A_{\alpha}g_{\alpha \beta,1}
\end{equation}
\begin{lemma}
Define the $0$-cochain $\AAA_1=(\AAA_{\alpha,1})$ in $\mathcal{E}nd(\EEE)\otimes\Omega^{1}_X (D)$ whose
matrix over $U_{\alpha}$ is $A_{\alpha,1}$ in the basis $e_{\alpha}$.
Then (\ref{equation 8}) implies:
\begin{equation}\label{equation 9}
(\check{d}\AAA_1)_{\alpha \beta}=\AAA_{\beta,1}-\AAA_{\alpha,1}=\ud a_{\alpha \beta}+[A_{\alpha}, a_{\alpha \beta}]
\end{equation}
\end{lemma}
\begin{proof}
Conjugate (\ref{equation 8}) by $g_{\alpha \beta}$:
\begin{equation}\label{equation 10}
g_{\alpha \beta}A_{\beta,1}g^{-1}_{\alpha \beta}=g^{-1}_{\beta \alpha}g_{\beta \alpha,1}\ud g_{\alpha \beta}g^{-1}_{\alpha \beta}+
\ud g_{\alpha \beta,1}g^{-1}_{\alpha \beta}+g_{\alpha \beta}g_{\beta \alpha,1}A_{\alpha}+A_{\alpha,1}+
A_{\alpha}g_{\alpha \beta,1}g^{-1}_{\alpha \beta}
\end{equation}

Then $g_{\alpha \beta}A_{\beta,1}g^{-1}_{\alpha \beta}$, $A_{\alpha,1}$ are the 
matrices of $\AAA_{\beta,1},\AAA_{\alpha,1}$ respectively in the basis $e_{\alpha}$;
we will also interprete all the remaining terms of (\ref{equation 10}) as matrices of some
sections of $\mathcal{E}nd(\EEE)\otimes\Omega^{1}(D)$. We have
\begin{equation}\label{equation 11}
g^{-1}_{\beta \alpha}g_{\beta \alpha,1}=a_{\beta \alpha}=-a_{\alpha \beta} ; g_{\alpha \beta,1}g^{-1}_{\beta \alpha}=a_{\alpha \beta},
\end{equation} so that

\begin{equation}\label{equation 12}
g_{\alpha \beta}g_{\beta \alpha,1}A_{\alpha}+A_{\alpha}g_{\alpha \beta,1}g^{-1}_{\alpha \beta}=[A_{\alpha}, a_{\alpha \beta}].\end{equation}
Next, $g_{\alpha \beta,1}=a_{\alpha \beta}g_{\alpha \beta}$ , so that
\begin{equation}\label{equation 13}
\ud g_{\alpha \beta,1}=\ud a_{\alpha \beta}g_{\alpha \beta}+a_{\alpha \beta}\ud g_{\alpha \beta} 
.\end{equation}
Further, by (\ref{equation 11}), \begin{equation}\label{equation 14}
g^{-1}_{\beta \alpha}g_{\beta \alpha,1}\ud g_{\alpha \beta}g^{-1}_{\alpha \beta}=-a_{\alpha \beta}\ud g_{\alpha \beta}g^{-1}_{\alpha \beta}
\end{equation}
Combining (\ref{equation 13}), (\ref{equation 14}), we obtain
\begin{equation}\label{equation 15}
g^{-1}_{\beta \alpha}g_{\beta \alpha,1}\ud g_{\alpha \beta}g^{-1}_{\alpha \beta}+
\ud g_{\alpha \beta,1}g^{-1}_{\alpha \beta}=-a_{\alpha \beta}\ud g_{\alpha \beta}g^{-1}_{\alpha \beta}+\ud a_{\alpha \beta}+
a_{\alpha \beta}\ud g_{\alpha \beta}g^{-1}_{\alpha \beta}=\ud a_{\alpha \beta}\end{equation}

Substituing (\ref{equation 12}), (\ref{equation 15}) into (\ref{equation 10}), we obtain (\ref{equation 9}).
\end{proof}
\begin{corollary}\label{corollary 1}
The pair $(\tilde{g}_{\alpha \beta}),(\tilde{\AAA_{\alpha}})$ defines a first order deformation
of $(\EEE, \nabla)$ if and only if the cochains $a=(a_{\alpha \beta})=(g_{\alpha \beta,1}g^{-1}_{\alpha \beta}),
\AAA_{\alpha,1}=A_{\alpha,1}$ (both given in the basis $e_{\alpha}$) satisfy the relations 
$\check{d}(a_{\alpha \beta})=0, \check{d}(\AAA_{\alpha,1})=( \ud a_{\alpha \beta}+[A_{\alpha}, a_{\alpha \beta}]).$
\end{corollary}
We will interprete the latter result in terms of the induced connection on $\mathcal{E}nd{(\EEE)}$.
As we saw, given  a connection $\nabla:\EEE\rar\EEE\otimes\Omega^{1}(D)$ on $\EEE$, we can define a connection
$\nabla_{\mathcal{E}nd(\EEE)}:\mathcal{E}nd(\EEE)\rar\mathcal{E}nd(\EEE)\otimes\Omega^{1}(D)$ by
$\nabla_{\mathcal{E}nd(\EEE)}(\phi)=\nabla\circ\phi-\phi\circ\nabla.$
If we represent $\phi$ by its matrix $M_{\alpha}$ in the basis $e_{\alpha}$, then
$\nabla_{\mathcal{E}nd(\EEE)}(\phi)=\ud M_{\alpha}+[A_{\alpha}, M_{\alpha}].$
Now, we can reformulate Corollary \ref{corollary 1} as follows.
\begin{proposition}\label{proposition 3}
The first order deformations of $(\EEE, \nabla)$ with fixed divisor of poles $D$ are classified
by the pairs $(a,\AAA_1)\in \check{C}^{1}(\mathfrak U, \mathcal{E}nd(\EEE))\times\check{C}^{0}(\mathfrak U,\mathcal{E}nd(\EEE)\otimes\Omega^{1}(D))$
such that \begin{equation}\label{equation 16}\check{d}(a)=0, \check{d}(\AAA_1)=\nabla_{\mathcal{E}nd(\EEE)}(a).\end{equation}
\end{proposition}

Now, let us assume in addition that the initial connection is integrable. Then the condition that
the deformed connection $(\tilde{\EEE}, \tilde{\nabla})$, given by the data $(a,\AAA_1)$ as in Proposition \ref{proposition 3}
, remains integrable, can be written in the form:
\begin{equation}\label{equation h}\ud A_{\alpha,1}=-A_{\alpha,1}\wedge A_{\alpha}-A_{\alpha}\wedge A_{\alpha,1},\end{equation}
or in an invariant form, $\nabla_{\mathcal{E}nd(\EEE)}(A_1)=0$.
We remark that here we consider $\nabla_{\mathcal{E}nd(\EEE)}$ extended to $\mathcal{E}nd(\EEE)\otimes\Omega^{\fatdot}(*D)$ in the
same way as was explained for $\nabla=\nabla_{\EEE}$.
\begin{proposition}\label{proposition 4}
The first order deformations of integrable connections $(\EEE, \nabla)$ with fixed divisor of poles $D$
are classified by the pairs $(a, \AAA_1)$ as above satisfying three relations
\begin{equation}\label{equation 17}
\check{d}(a)=0, \check{d}(\AAA_1)=\nabla_{\mathcal{E}nd(\EEE)}(a), \nabla_{\mathcal{E}nd(\EEE)}(\AAA_1)=0.
\end{equation}
\end{proposition}

\subsection{Hypercohomology}
Let $K^{\fatdot}=(K^p,d_K)$ be a complex of sheaves over $X$, and $\mathfrak U=(U_{\alpha})$ a sufficiently fine open
covering of $X$. The $\check{C}$ech complex of $K^{\bullet}$ is the double complex \begin{equation}\label{equation 18} 
(\check{C}^{p}(\mathfrak U, K^q), \check{d}, (-1)^{p} d_K).\end{equation}
The hypercohomology group $\HH^{i}(X, K^{\bullet})$ is by definition the $i$-th cohomology of the simple complex $(L^{\bullet}, D)$
associated to (\ref{equation 18}):
$$L^{n}=\oplus_{p+q=n} \check{C}^p(\mathfrak U, K^q), D_{\vert_{\check{C}^p(\mathfrak U, K^q)}}=\check{d}+(-1)^{p} d_K,$$
$$\HH^{i}(X,K^{\fatdot}):=H^{i}(L^{\fatdot}, D).$$
A hypercohomology class $c\in\HH^{i}(X,K^{\bullet})$ is represented by a cocycle $c\in L^i$,
$c=(\dots,c^{p-1, q+1}, c^{p, q},c^{p+1, q-1}, \dots)$, where $p+q=i$, and the cocycle condition is
$(\dots, \check{d}c^{p-1, q+1}+(-1)^{p}d_K c^{p,q}=0, \check{d}c^{p,q}+(-1)^{p+1}d_K c^{p+1,q-1}=0,\dots).$
A cocycle $(c^{p,q})_{p+q=n}$ is a coboundary if there exists a cochain $(b^{p,q})_{p+q=n-1}$ such that
$$c^{p,q}=\check{d}b^{p-1,q}+(-1)^{p}d_K b^{p,q-1}.$$
We denote the $i$-cocycles $\check{Z}^{i}(\mathfrak U, K^{\bullet})$ and the $i$-coboundaries $\check{B}^{i}(\mathfrak U, K^{\bullet})$,
so that $$\HH^{i}(X, K^{\fatdot})=\check{Z}^{i}(\mathfrak U, K^{\fatdot})/\check{B}^{i}(\mathfrak U, K^{\fatdot}).$$
Let now come back to the setting of Proposition \ref{proposition 3}. Define the two-term complex of sheaves
\begin{equation}\label{equation c} \CCC^{\fatdot}=[\CCC^0\rar\CCC^1],\end{equation} where $\CCC^0=\mathcal{E}nd(\EEE)$, $\CCC^1=\mathcal{E}nd(\EEE)\otimes\Omega^1(D)$, and differential $d_{\CCC}=\nabla_{\mathcal{E}nd(\EEE)}.$
Then the equations (\ref{equation 16}) express the fact that  $(a, \AAA_1)\in \check{Z}^{1}(\mathfrak U, \CCC^{\fatdot}).$
Changing the bases $e_{\alpha}$ over $V_1=\Spec k[\epsilon]/(\epsilon^2)$ by the rule $\tilde{e}_{\alpha}=e_{\alpha}(1+\epsilon h_{\alpha})$,
where $h=(h_{\alpha})\in\check{C}^{0}(\mathfrak U, \mathcal{E}nd(\EEE))=\check{C}^{0}(\mathfrak U, \CCC^0)$, we obtain the transformation rule
of the cocycle $(a, \AAA_1)$ in the following form:
$(a,\AAA_1)\rar(a+\check{d}h, \AAA_1+d_{\CCC}h)$, so that isomorphic first order deformations differ by a $1$-coboundary.
We deduce:
\begin{theorem}\label{theorem 1.1}
Let $X$ be a complete scheme of finite type over $k$ or a complex space (then $k=\CC$). Let $\EEE$ be a vector bundle on $X$
and $\nabla$ a rational (or meromorphic) connection on $\EEE$ with divisor of poles $D$.
Then the isomorphism classes of first order deformations of $(\EEE, \nabla)$ with fixed divisor of poles are classified by 
$\HH^{1}(X, \CCC^{\fatdot})$.
\end{theorem}
In order to characterize the first order deformations of integrable connections, we introduce two other
complexes:
$$\RRR^{\fatdot}=[\mathcal{E}nd(\EEE)\rar\mathcal{E}nd(\EEE)\otimes\Omega^{1}(D)\rar\mathcal{E}nd(\EEE)\otimes\Omega^{2}(*D)\rar\dots]$$
with differential $d_{\RRR}=\nabla_{\End(\EEE)}$, and
\begin{equation} \FFF^{\fatdot}=[\FFF^0\xymatrix@1{\ar[r]^{d_{\FFF}}&\FFF^{1}]},\end{equation}
where $\FFF^0=\mathcal{E}nd(\EEE)$, $d_{\FFF}=\nabla_{\mathcal{E}nd(\EEE)}$, and $\FFF^{1}=\ker(\mathcal{E}nd(\EEE)\otimes\Omega^{1}(D))\rar\mathcal{E}nd(\EEE)\otimes\Omega^{2}(*D))$.
It is easy to see that these complexes have the same $1$-cocycles and $1$-coboundaries, so that
$$\HH^{1}(X, \FFF^{\fatdot})=\HH^{1}(X, \RRR^{\fatdot}).$$
The formulas (\ref{equation 14}) express the fact that the pair $(a, \AAA_1)$ is a $1$-cocycle in 
either one of the complexes $\FFF^{\fatdot}, \RRR^{\fatdot}$.
\begin{theorem}\label{theorem 1.2}
Let $X$ be a scheme of finite type over $k$ or a complex space (then $k=\CC$). Let $\EEE$ a vector bundle on $X$
and $\nabla$ a rational (or meromorphic) integrable connection on $\EEE$ with fixed divisor of poles $D$.
Then the isomorphism classes of first order deformations of $(\EEE, \nabla)$ in the class of integrable
connections with fixed divisor of poles $D$ are classified by 
$$\HH^{1}(X, \FFF^{\fatdot})=\HH^{1}(X, \RRR^{\fatdot}).$$
\end{theorem}
\section{Obstructions}\label{Obstruc}
\subsection{First obstruction}
Let $X, \EEE, \nabla,(a, \AAA_1)$ be as in Theorem \ref{theorem 1.1}, and let $(\EEE_1,\nabla_1)$ be the
first order deformation of $(\EEE, \nabla)$ over $V_1$ associated to $(a,\AAA_1)$. We want to determine the obstruction
to extend $(\EEE_1,\nabla_1)$ to $(\EEE_2,\nabla_2)$ over $V_2=\Spec k[\epsilon]/(\epsilon^3)$.
As before, we only consider deformations with fixed divisor of poles $D$.
We search for the extended data
$$G_{\alpha \beta}=(1+\epsilon a_{\alpha \beta}+\epsilon^2 a_{\alpha\beta,2})g_{\alpha \beta}=g_{\alpha \beta}+\epsilon g_{\alpha \beta,1}+\epsilon^2 g_{\alpha \beta,2}$$
$$\tilde{A_{\alpha}}=A_{\alpha}+\epsilon A_{\alpha,1}+\epsilon^2 A_{\alpha,2},\ \AAA_{\alpha,1}=A_{\alpha,1},$$
with respect to the basis $e_{\alpha}$.
We assume that they satisfy the cocycle condition modulo $\epsilon^2$. Then the cocycle condition modulo $\epsilon^3$ has two
counterparts: the one expressing the extendability of $\EEE_1$, which we have already treated in Section $2$, and the other expressing the
extendability of the connection. The latter has the following form: 
\begin{eqnarray}\label{eqnarray.3..}% 
A_{\beta,2}=g_{\beta \alpha,2}\ud g_{\alpha \beta}+g_{\beta \alpha,1}\ud g_{\alpha \beta,1}+g_{\beta \alpha}\ud g_{\alpha \beta,2}& \\ \nonumber
+g_{\beta \alpha,2}A_{\alpha}g_{\alpha \beta}+g_{\beta \alpha}A_{\alpha,2}g_{\alpha \beta}+g_{\beta \alpha}A_{\alpha}g_{\alpha \beta,2}& \\ \nonumber+ g_{\beta \alpha,1}A_{\alpha,1}g_{\alpha \beta}+g_{\beta \alpha,1}A_{\alpha}g_{\alpha \beta,1}+g_{\beta \alpha}A_{\alpha,1}g_{\alpha \beta,1}
\end{eqnarray}
Introduce the cochain $\AAA_2\in\check{C}^0(\mathfrak U, \mathcal{E}nd(\EEE)\otimes\Omega^1 (D))$ given  over $U_{\alpha}$ by the matrix $A_{\alpha,2}$ in the basis $e_{\alpha}$. By transformations similar to those used in the proof of $(10)$, and in using formulas $(22)$ and
$a_{\beta \alpha,2}-(a_{\alpha \beta,1})^{2}+a_{\alpha \beta,2}=0$, we reduce (\ref{eqnarray.3..}) to the following equation:
\begin{eqnarray}\label{eqnarray.4..}% 
\nabla_{\mathcal{E}nd(\EEE)}(a_{\alpha \beta,2})-\nabla_{\mathcal{E}nd(\EEE)}(a_{\alpha \beta,1})a_{\alpha \beta,1}-[a_{\alpha \beta,1},\AAA_{\beta,1}]& \\ \nonumber =\nabla_{\mathcal{E}nd(\EEE)}(a_{\alpha \beta,2})+\AAA_{\alpha,1}a_{\alpha \beta,1}-a_{\alpha \beta,1}\AAA_{\beta,1} =\AAA_{\beta,2}-\AAA_{\alpha,2}
\end{eqnarray}
Let us denote \begin{equation}\label{equation 19}k_{\alpha \beta}=\nabla_{\mathcal{E}nd(\EEE)}(a_{\alpha \beta,2})+\AAA_{\alpha,1}a_{\alpha \beta,1}-a_{\alpha \beta,1}\AAA_{\beta,1}.\end{equation}
We consider $k=(k_{\alpha \beta})$ as a cochain in $\check{C}^1(\mathfrak U, \mathcal{E}nd(\EEE)\otimes\Omega^1 (D)).$
\begin{lemma}
$k$ is a skew-symmetric cocycle.
\end{lemma}
\begin{proof}
A straightforward calculation using the relations
\begin{eqnarray}\label{eqnarray.5..}%
a_{\alpha \beta,2}+a_{\beta \gamma,2}+a_{\gamma \alpha,2}=-a_{\alpha \beta,1}a_{\beta \gamma,1}-a_{\beta \gamma,1}a_{\gamma \alpha,1}-a_{\alpha \beta,1}a_{\gamma \alpha,1}& \\ \nonumber \end{eqnarray}
and $\nabla_{\mathcal{E}nd(\EEE)}(XY)=\nabla_{\mathcal{E}nd(\EEE)}(X)Y+Y\nabla_{\mathcal{E}nd(\EEE)}(X)$,
for any local sections $X,Y$ of $\mathcal{E}nd(\EEE)$
\end{proof}

\begin{proposition}\label{proposition 1.2}
Let $(a,\AAA_1)\in\check{Z}^{1}(\mathfrak U, \CCC^{\fatdot})$, and let $(\EEE_1,\nabla_1)$ be the deformation 
of $(\EEE, \nabla)$ over $V_1$ defined by $(a, \AAA_1)$. Then $(\EEE_1,\nabla_1)$ extends to a deformation
$(\EEE_2,\nabla_2)$ over $V_2$ if and only if the following two conditions are verified:\\
$(i)$ The Yoneda square $[a_1]\circ[a_1]\in H^{2}(X,\mathcal{E}nd(\EEE))$ vanishes.\\
$(ii)$ Provided $(i)$ holds, let $a_2=(a_{\alpha \beta,2})\in\check{C}^{1}(\mathfrak U,\mathcal{E}nd(\EEE))$
be a solution of (\ref{eqnarray.5..}), and let $k=(k_{\alpha \beta})$ be the cocycle (\ref{equation 19}) determined by this choice of $a_2$. Then
$[k]\in H^{1}(X,\mathcal{E}nd(\EEE)\otimes\Omega^{1}(D))$ vanishes.
\end{proposition}
The expression $\AAA_{\alpha,1}a_{\alpha \beta,1}-a_{\alpha \beta,1}\AAA_{\beta,1}$ entering (\ref{equation 19}) is a component
$c^{1,1}$ of the $\check{C}$ech cocycle $(c^{1,1}, c^{2,0})\in\check{Z}^{2}(\mathfrak U, \CCC^{\fatdot})$ representing 
the Yoneda square $[a_1,\AAA_{1}]\circ[a_1,\AAA_{1}]$. The other component is $c^{2,0}_{\alpha \beta \gamma}=a_{\alpha \beta,1}
a_{\beta \gamma,1}+a_{\beta \gamma,1}a_{\gamma \alpha,1}+a_{\alpha \beta,1}a_{\gamma \alpha,1}.$
Hence we have:
\begin{proposition}
Under the assumptions of Prop. (\ref{proposition 1.2}), $(\EEE_1, \nabla_1)$ extends to $(\EEE_2,\nabla_2)$ over $V_2$
with fixed divisor of poles $D$ if and only if the Yoneda square $[a_1,\AAA_{1}]\circ[a_1,\AAA_{1}]$ vanishes in $\HH^{2}(X,\CCC^{\fatdot})$.
\end{proposition}
\subsection{Infinitesimal deformations of the Atiyah class}
We fix a vector bundle $\EEE$ on $X$ given by a cocycle $g_{\alpha \beta}$. Recall that we defined the Atiyah class of $\EEE$
as the cohomology class of the cocycle $\GGG_{\alpha \beta}=\ud g_{\alpha \beta}g^{-1}_{\alpha \beta}$ (here $\GGG_{\alpha \beta}$
is considered as a section of $\mathcal{E}nd(\EEE)\otimes\Omega^{1}(D)$ given by the matrix $\ud g_{\alpha \beta}g^{-1}_{\alpha \beta}$
in the basis $e_{\alpha}$).  

If $\EEE_i$ is an extension of $\EEE$ (as a vector bundle) to $X\times V_i$, where $V_i=\Spec k[\epsilon]/(\epsilon^{i+1})$,
then we can define the Atiyah class $\At(\EEE_i)\in H^{1}(X,\mathcal{E}nd(\EEE_i)\otimes\Omega^1)$ by the cocycle $\GGG_{i,\alpha \beta}=
\ud g_{i,\alpha \beta}g^{-1}_{i,\alpha \beta}$, where $(g_{i,\alpha \beta})$ is a cocycle defining $\EEE_i$, $g_{i,\alpha \beta}\in\Gamma(U_{\alpha \beta}, M_r(\OOO_X)\otimes k[\epsilon]/(\epsilon^{i+1}))$.
The following assertion is obvious.
\begin{lemma}
Assume that $\EEE$ admits a connection $\nabla$ with fixed divisor of poles $D$. Then $\nabla$ extends to a connection $\nabla_i$ on $\EEE_i$ with
fixed divisor of poles $D$ if and only if the image $\At^D(\EEE_i)$ of $\At(\EEE_i)$ in $H^{1}(X,\mathcal{E}nd(\EEE_i)\otimes\Omega^1 (D))$ is zero.
\end{lemma}
\begin{corollary}
Let $j>0$, and assume $\EEE$ extends to a vector bundle $\EEE_j$ over $X\times V_j$.
For any $i\geq 0, i\leq j$, denote by $\EEE_i$ the restriction of $\EEE_j$ to $X\times V_i$.
The following assertions hold:\\
$(i)$ if $\nabla_j$ is a connection with fixed divisor $D$ of poles on $\EEE_j$, then $\nabla_i=\nabla_{j}\vert_{\EEE_i}$ is
such a connection on $\EEE_i$. Thus $\At^D (\EEE_j)=0\Rightarrow \At^D(\EEE_i)=0 (i\leq j$).\\
$(ii)$ Let $\At^D (\EEE_j)=0$. Introduce the natural restriction map 
$$\res_{ji}:H^{0}(\mathcal{E}nd(\EEE_j)\otimes\Omega^1 (D))\rar  H^{0}(\mathcal{E}nd(\EEE_i)\otimes\Omega^1 (D))$$
$$\phi\mapsto \phi\otimes k[\epsilon]/(\epsilon^{i+1})$$
Then any connection with fixed divisor of poles $D$ on $\EEE_i$ extends to such a connection on $\EEE_j$ if and
only if $\res_{ji}$ is surjective. 
\end{corollary}
\begin{proof}
$(i)$ is obvious. To prove $(ii)$, we use the following observation:
for two connections $\nabla_j, \nabla'_j$ on $\EEE_j$ with fixed divisor $D$ of poles, the difference
$\nabla_j-\nabla'_j$ is an element of $H^{0}(\mathcal{E}nd(\EEE_j)\otimes\Omega^1 (D))$ and $(\nabla_j-\nabla'_j)\vert_{\EEE_i}=
\res_{ji}(\nabla_j-\nabla'_j)\in H^{0}(\mathcal{E}nd(\EEE_i)\otimes\Omega^1 (D))$.
\end{proof}

In this Corollary, it is possible that both $\EEE_i, \EEE_j$ admit 
connections with fixed divisor of poles $D$, but not every connection with the same $D$
on $\EEE_i$ extends to such a connection on $\EEE_j$.
To produce an example, set $D=0, i=0, j=1, X$ an elliptic curve, $\EEE=\OOO_X^{\oplus 2}$.
Define $\EEE_1$ as a nontrivial extension of vector bundles
\begin{eqnarray}\label{eqnarray a}%
 0\ra\OOO_{X\times V_1}\xymatrix@1{\ar[r]^{\mu}&}\EEE_1\xymatrix@1{\ar[r]^{\nu}&}\OOO_{X\times V_1}\ra 0
\end{eqnarray}
Such extensions are classified by $\Ext^{1}(\OOO_{X\times V_1},\OOO_{X\times V_1})=H^{1}(\OOO_{X\times V_1})\simeq k[\epsilon]/(\epsilon^2)$,
and we choose an extension class in the form $\epsilon[f]$, so that the extension is trivial modulo $\epsilon^2$.
We can describe $[f]$ and the associated extension explicitly as follows.
Let $\mathfrak U=\{U_{+-}\}$ be an open covering of X, and $f\in\Gamma(U_{\pm},\OOO_X)$ 
a function whose cohomology class $[f]$ generates $H^{1}(X,\OOO_X)$.
Let $e_{\pm}=(e_{\pm 1},e_{\pm 2})$ be a basis of $\EEE\vert_{U_{+-}}$, and define the transition matrix
over $U_{+-}$ by 
\begin{eqnarray}
\left( \begin{array}{cc} 1 & \epsilon f \\ 0 & 1
\end{array}\right).
\end{eqnarray}
Define the maps $\mu, \nu$ in (\ref{eqnarray a}) by
$\mu:1\mapsto e_{\pm 1}$, $\nu:(e_{\pm 1},e_{\pm 2})\mapsto (0,1)$.
To be more explicit, we will give $X$ by the Legendre equation $$y^2=x(x-1)(x-t)\ (t\in k\setminus\{0,1\}),$$
and define an open covering $\mathfrak U$ of $X$ by $U_{+}=X\setminus\{\infty\}, U_{-}=X\setminus\{0\}$.  
Then we can choose $f=\frac{y}{x}$ as a function having two simple poles at $0$ and $\infty$ and no other singularities.
The Residue Theorem implies that it is impossible to represent $f$ as the difference of two functions, 
one regular on $U_+$ and the other on $U_-$, so the cohomology class of $f$ considered as a $\check{C}$ech cocycle
of the covering $\mathfrak U$ with coefficients in $\OOO_X$ is nonzero.
We now verify that $\At(\EEE_1)=0$. It is represented by the cocycle
\begin{eqnarray}
\ud g_{+-}g^{-1}_{+-}=\left( \begin{array}{cc} 0 & \epsilon df \\ 0 & 0
\end{array}\right),
\end{eqnarray}
and $$df=\ud(\frac{y}{x})=\frac{\ud y}{x}-y\frac{\ud x}{x^2}=\omega_+ -\omega_-,$$
where $$\omega_+=2\frac{\ud y}{x}-y\frac{\ud x}{x^2},\ \omega_-=\frac{\ud y}{x},$$ $\omega_+$ (resp. $w_-$) being regular on $U_+$ (resp. $U_{-}$).
Hence,
\begin{eqnarray}
\ud g_{+-}g^{-1}_{+-}=\left( \begin{array}{cc} 0 & \epsilon\omega_+\\ 0 & 0 \end{array}\right)
-\left( \begin{array}{cc} 0 & \epsilon\omega_-\\ 0 & 0 \end{array}\right)
\end{eqnarray}
is a \v Cech coboundary, and $\At(\EEE_1)=0$.
Thus $\EEE_1$ has a regular connection.\\
Now, we will show that the map $\res_{10}$ defined in the last corollary is not surjective, so not
every regular connection on $\EEE$ extends to a regular connection on $\EEE_1$.
We remark that in our case $\Omega^{1}_X$ is trivial, $D=0$, so $\res_{10}$  is just the restriction map
$\res_{10}:H^{0}(\mathcal{E}nd(\EEE_1))\rar H^{0}(\mathcal{E}nd(\EEE_0)).$
Consider $\EEE_1$ as an extension of another kind:
$$0\rar\epsilon\EEE\rar\EEE_1\rar\EEE\rar 0,$$
where $\epsilon\EEE\simeq\OOO_X^{\oplus{2}}$ and $\EEE\simeq\EEE_1/\epsilon\EEE\simeq\OOO_X^{\oplus{2}}$.
Apply to it $\mathcal{H}om(\EEE_1, .)$(the Hom-sheaf as $\OOO_{X\times V_1}$-modules):
$$0\rar \mathcal{H}om(\EEE_1,\EEE)\rar\mathcal{E}nd{(\EEE_1)}\rar\mathcal{H}om(\EEE_1, \EEE)\rar 0.$$
As $\EEE\simeq\OOO_X^{\oplus 2}$, the first and the third terms of the last triple
are described as follows:
$$\mathcal{H}om(\EEE_1,\EEE)\simeq\mathcal{E}nd{(\EEE)}=M_2(\OOO_X).$$
Take an element in $H^0(\mathcal{H}om(\EEE_1, \EEE))\simeq M_2(k)$
given by the matrix \begin{eqnarray}
\left( \begin{array}{cc} 0 & 0 \\ 0 & 1 
\end{array}\right),
\end{eqnarray}
(as above, $\EEE_1, \EEE$ are trivialized by the bases $e_{\pm}=(e_{\pm 1},e_{\pm 2})).$
We will see that it is not in the image of the restriction map $\res_{1,0}$.\\
Indeed, assume there is a lift of \begin{eqnarray}\left( \begin{array}{cc} 0 & 0 \\ 0 & 1 
\end{array}\right)
\end{eqnarray} to $H^0(\mathcal{E}nd(\EEE_1))$. Then it is given in the basis $e_{+}$ by a matrix of the form 
\begin{eqnarray}
A_+=\left( \begin{array}{cc} 0 & 0 \\ 0 & 1 
\end{array}\right)+\epsilon B,
\end{eqnarray}
$B\in M_2(k[U_+])$. Transforming it to the basis $e_-$, we obtain the matrix
\begin{eqnarray}
A_-=\left( \begin{array}{cc} 0 & -\epsilon f \\ 0 & 1 
\end{array}\right)+\epsilon B,
\end{eqnarray}
which has to be regular in $U_-$. Thus $\epsilon f=\epsilon b_{12}-a_{-12}$, where $b_{12}$ is
regular in $U_+$ and $a_{-12}$ is regular in $U_-$.
This contradicts the fact that $f$ is not a $\check{C}$ech coboundary in $\check{C}(\mathfrak U, \OOO_X)$,
and this ends the proof.
\subsection{Kuranishi space for deformations of connections}
\begin{theorem}\label{theorem 1.3}
Let $X$ be a complete scheme of finite type over $k$ or a complex space (in which case $k=\CC$), $\CCC^{\fatdot}$
the $2$-term complex of sheaves on $X$ defined by (\ref{equation c}),
$W=\HH^{1}(X,\CCC^{\fatdot})$,$(\delta_1\dots,\delta_N)$ a basis of $W$ and $(t_1,\dots,t_N)$ the dual coordinates on $W$.
Let $W_k$ denote the $k$-th infinitesimal neighborhood of $0$ in $W$, and $(\EEE_1, \nabla_1)$ the universal first order
deformation over $X\times W_1$ of a connection $(\EEE,\nabla)$ on $X$ with fixed divisor of poles $D$. Then there exists a formal power series
$$f(t_1,\dots,t_N)=\sum_{k=2}^{\infty} f_{k}(t_1\dots,t_N)\in \HH^{2}(X, \CCC^{\fatdot})[[t_1,\dots,t_N]],$$
where $f_k$ is  homogeneous of degree $k$ ($k\geq 2$), with the following property.
Let $I$ be the ideal of $k[[t_1,\dots,t_N]]$ generated by the image of the map $f^*:\HH^{2}(X,\CCC^{\fatdot})^{*}\rar k[[t_1,\dots,t_N]]$, adjoint to $f$.
Then for any $k\geq 2$, the pair $(\EEE_1,\nabla_1)$ extends to a connection $(\EEE_k,\nabla_k)$ on $X\times V_k$, where
$V_k$ is the closed subscheme of $W_k$ defined by the ideal $I\otimes k[[t_1,\dots,t_N]]/(t_1,\dots,t_N)^{k+1}.$
\end{theorem}
\begin{proof}
We will start by fixing a particular choice of coordinates $(t_1,\dots,t_N)$, coming from the spectral sequence 
$E_1^{p,q}=H^{q}(\CCC^p)\Rightarrow\HH^{p+q}(\CCC^{\bullet}).$ The latter is supported on two vertical strings
$p=0$ and $p=1$ (see Fig. \ref{figu}).
\begin{figure}[h]
%\begin{figure}[t]
\begin{picture}(0,250)(50,95)
%(18,5)(2,17)
\includegraphics{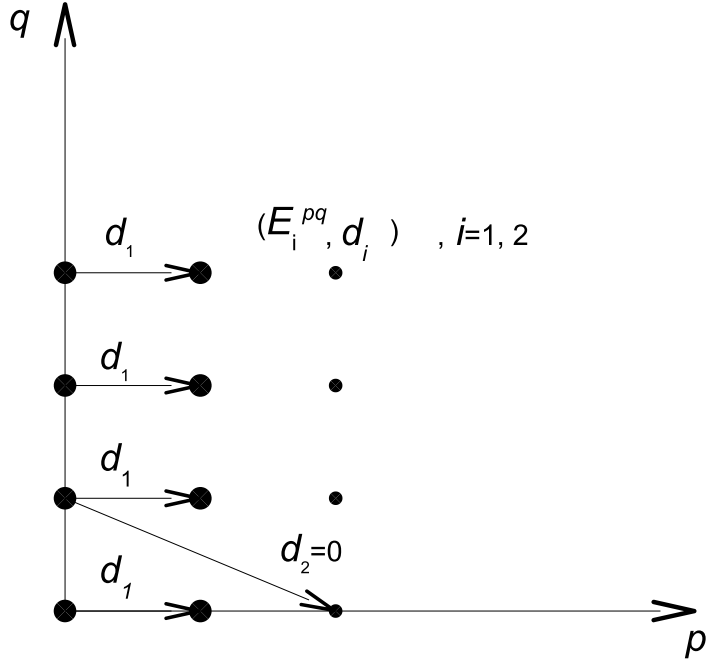}
\end{picture}
\caption{The spectral sequence is supported on $2$ vertical strings $p=0, p=1$.}
\label{figu}
\end{figure}

Thus the spectral sequence degenerates in the second term $E_2$, and we have the long
exact sequence
\arraycolsep=0.ex
\begin{eqnarray*}\label{eqnarray}
0&\lra&\HH^0(X,\CCC^{\fatdot})\lra H^{0}(X,\mathcal{E}nd(\EEE))\xymatrix@1{\ar[r]^{d_1}&}H^{0}(X,\mathcal{E}nd(\EEE)
\otimes\Omega^{1}_X (D))\\
&\lra&\HH^{1}(X,\CCC^{\fatdot})\lra H^{1}(X,\mathcal{E}nd(\EEE)) \xymatrix@1{\ar[r]^{d_1}&}H^{1}(X,\mathcal{E}nd(\EEE)\otimes\Omega^{1}_X (D))\\
&\lra&\HH^{2}(X,\CCC^{\fatdot})\lra H^{2}(X,\mathcal{E}nd(\EEE))\xymatrix@1{\ar[r]^{d_1}&}H^{2}(X,\mathcal{E}nd(\EEE)\otimes\Omega^{1}_X (D))\ra\dots,
\nonumber\end{eqnarray*}

We deduce the exact triple
$$0\rar W'\rar W\rar W''\rar 0,$$
with $$W'=\frac{H^{0}(X, \End(\EEE)\otimes\Omega^{1}_X (D))}{\im d_1}, W=\HH^{1}(X,\CCC^{\fatdot}),$$ $$W''=\ker(H^{1}(X,\End(\EEE))
\rar H^{1}(X, \End(\EEE)\otimes\Omega^{1}_X (D)).$$
Let $N'=\dim W'$, $N''=\dim W''$; choose $t_1,\dots,t_N$ in such a way that $s_1=t_{N'+1},\dots,s_{N''}=t_{N'+N''} (N=N'+N''),$
are coordinates on $W''$ and $t_1,\dots,t_{N'}$ restrict to $W'$ as coordinates on $W'$.
We will construct by induction on $k\geq 0$ the homogeneous forms
\begin{eqnarray}\label{eqnarray.6..}%
G_{\alpha \beta, k}(s_1,\dots,s_{N''})\in\Gamma(U_{\alpha \beta},\mathcal{E}nd(\EEE))\otimes k[s_1,\dots,s_{N''}], & \\ \nonumber
F_{\alpha \beta \gamma, k}(s_1,\dots,s_{N''})\in\Gamma(U_{\alpha \beta \gamma},\mathcal{E}nd(\EEE))\otimes k[s_1,\dots,s_{N''}], & \\ \nonumber
\bar{f}_{k}(s_1,\dots,s_{N''})\in H^{2}(X,\mathcal{E}nd(\EEE))\otimes k[s_1,\dots,s_{N''}], & \\ \nonumber
A_{\alpha ,k}(t_1\dots,t_N)\in\Gamma(U_{\alpha},\mathcal{E}nd(\EEE)\otimes\Omega^{1}_X (D))\otimes k[t_1,\dots,t_N], & \\ \nonumber
\kappa_k(t_1\dots,t_N)\in H^1(X,\mathcal{E}nd(\EEE)\otimes\Omega^{1}_X (D))\otimes k[t_1,\dots,t_N], & \\ \nonumber
K_{\alpha \beta, k}(t_1\dots,t_N)\in\Gamma(U_{\alpha \beta},\mathcal{E}nd(\EEE)\otimes\Omega^{1}_X (D))\otimes k[t_1,\dots,t_N]
\end{eqnarray} 
with the following properties:\\
$(i)$ $G_{\alpha \beta, 0}=g_{\alpha \beta}$, and $A_{\alpha,0}$ define $\EEE$ and resp. $\nabla$ with respect to the
local trivializations $e_{\alpha}$ of $\EEE$ on $U_{\alpha}$.\\ 
$(ii)$ $\bar{f_k}=0, F_{\alpha \beta \gamma, k}=0$ for $k=0,1$, and $K_{\alpha \beta,0}=0$.\\
$(iii)$ For each $k\geq 1$, let $\bar{f}^{(k)}=\sum_{i\leq k}\bar{f_i}$, and let $\bar{I}^{(k+1)}$ be the ideal generated by $(s_1,\dots,s_{N''})^{k+2}$ and the image of the adjoint map $\bar{f}^{(k)*}:H^{2}(X, \mathcal{E}nd(\EEE))^*\rar k[s_1,\dots,s_{N''}]$. Then $(F_{\alpha \beta \gamma, k+1})$ is a cocycle modulo $\bar{I}^{(k+1)}$ and $\bar{f}_{k+1}$ is a lift to $W''\otimes k[s_1,\dots,s_{N''}]$ of the cohomology class
$$[(F_{\alpha \beta \gamma, k+1}\mod \bar{I}^{(k+1)})]\in W''\otimes k[s_1,\dots,s_N'']/\bar{I}^{(k+1)}.$$\\
$(iv)$ For any $k\geq 1$, set $G^{(k)}_{\alpha \beta}=\sum_{i\leq k}G_{\alpha \beta, i}$. Then \begin{equation}\label{equation 21}G^{(k)}_{\alpha \beta}G^{(k)}_{\beta \gamma}G^{(k)}_{\gamma \alpha}\equiv(1+F_{\alpha \beta \gamma, k+1})\mod \bar{I}^{(k+1)}.\end{equation}\\
$(v)$ For each $k\geq 1$, set $\kappa^{(k)}=\sum_{i\leq k} \kappa_i$, and let $J^{(k+1)}$ be the ideal generated by $(t_1,\dots,t_N)^{k+2}$ and by the image of the adjoint map $\kappa^{(k)*}:H^{1}(X, \mathcal{E}nd(\EEE)\otimes\Omega^{1}(D))^*\rar k[t_1,\dots,t_N]$. Then $(K_{\alpha \beta, k+1})$ is a cocycle modulo $J^{(k+1)}+\bar{I}^{(k+2)}$ and $\kappa_{k+1}$ is a lift of the cohomology class $$[(K_{\alpha \beta, k+1}\mod (J^{k+1}+\bar{I}^{(k+2)})]\in H^1(X, \mathcal{E}nd(\EEE)\otimes\Omega^1(D))\otimes k[[t_1,\dots,t_N]]/(J^{k+1}+\bar{I}^{(k+1)})$$ in  $H^1(X, \mathcal{E}nd(\EEE)\otimes\Omega^1(D))\otimes k[t_1,\dots,t_N]$.\\
$(vi)$ For any $k\geq 0$, set   $A^{(k)}_{\alpha}=\sum_{i\leq k} A_{\alpha,i}$. Then \begin{equation}\label{equation 22}
K_{\alpha \beta, k+1}\equiv\ud G^{(k+1)}_{\alpha \beta}- G^{(k+1)}_{\alpha \beta}A^{(k)}_{\beta}+A^{(k)}_{\alpha} G^{(k+1)}_{\alpha \beta} \mod (J^{k+1}+\bar{I}^{(k+2)}).\end{equation}
In these properties, $G^{(k)}_{\alpha \beta}$ is considered as an endomorphism of $\EEE_k$ over $U_{\alpha \beta}\times V_k$
given by its matrix with respect to two bases: $e_{\alpha}$ for the source, $e_{\beta}$ for the target, where $\EEE_k$ 
is the vector bundle over $X\times V_k$ defined by the $1$-cocycle $(G^{(k)}_{\alpha \beta})$. Similarly $(A^{(k)}_{\alpha})$ is
understood as a $1$-cochain with values in $\mathcal{E}nd(\EEE_k)\otimes\Omega^{1}(D)$, and in formula (\ref{equation 22}), $A^{(k)}_{\alpha}$
(resp $A^{(k)}_{\beta}$) is represented by its matrix in the basis $e_{\alpha}$ (resp $e_{\beta}$).
The base changes $G_{\alpha \beta,k+1}$ acting on both sides of (\ref{equation 22}), reduce to $G_{\alpha \beta,0}$,
since the only nonzero terms in (\ref{equation 22}) are of degree $k+1$, and everything is reduced modulo $(t_1,\dots,t_N)^{k+2}$.
Thus (\ref{equation 22}) defines $(K_{\alpha \beta, k+1})$ as a $1$-cochain with values in $\mathcal{E}nd(\EEE)\otimes\Omega^{1}(D)$. Going over to the proof, we first remark that $G_{\alpha \beta,0}, A_{\alpha,0}$ are already known, and we have to indicate the choice of $G_{\alpha \beta,k}, A_{\alpha,k}$ inductively on $k\geq 0$, the other data $F_{\alpha \beta \gamma, k},\bar{f}_{k},
K_{\alpha \beta, k},\kappa_{k}$ being recovered via formulas (\ref{equation 21}),(\ref{equation 22}).
To initialize the induction, first look at (\ref{equation 21}) with $k=0$. Then $F_{\alpha \beta \gamma, 1}=0$
by $(ii)$, which implies
\begin{equation}\label{equation 23} G_{\alpha \beta,1}G_{\beta \gamma,0}G_{\gamma \alpha,0}+G_{\alpha \beta,0}G_{\beta \gamma,1}G_{\gamma \alpha,0}
+G_{\alpha \beta,0}G_{\beta \gamma,0}G_{\gamma \alpha,1}=0\end{equation}
The latter equation expresses the fact that $(G_{\alpha \beta,1})$ is a $1$-cocycle with values in $\mathcal{E}nd(\EEE)\otimes (W'')^*$.
As in Section 2, we can write $G_{\alpha \beta,1}=\sum a^{(i)}_{\alpha \beta}g_{\alpha\beta}s_i$, where $[(a^{(i)}_{\alpha \beta})]$ for $i=1,\dots,N''$ form the basis of $W''$ dual to $s_1,\dots,s_{N''}$.
Here and further on, we adopt the following convention:
all the $G_{\alpha \beta,k}$ (resp. $G^{(k)}_{\alpha \beta})$ are regarded as $1$-cochains with values in $\mathcal{E}nd(\EEE)$
(resp. $\mathcal{E}nd(\EEE_k)$) given by matrices with respect to two bases: $e_{\alpha}$ for the source, $e_{\beta}$ for the target.
We denote by $\EEE_k$ the vector bundle over $X\times V_k$ defined by the cocycle $G^{(k)}_{\alpha \beta}$.

Hence, looking at the first term $G_{\alpha \beta,1}G_{\beta \gamma,0}G_{\gamma \alpha,0}$ of the sum in (\ref{equation 23}), we see that it
represents the matrix of $G_{\alpha \beta,1}$ with respect to one and the same basis $e_{\alpha}$ for the source
and the target. The same applies to the other two summands in (\ref{equation 23}), thus (\ref{equation 23}) is the cocycle condition
$$ a_{\alpha \beta}+a_{\beta \gamma}+a_{\gamma \alpha}=0$$
put down via matrices of the three summands in the basis $e_{\alpha}$.

We will adopt the same convention for cochains with values in $\mathcal{E}nd(\EEE)\otimes\Omega^1 (D)$ or 
in $\mathcal{E}nd(\EEE_k)\otimes\Omega^1 (D)$. The $A_{\alpha,k}$ (resp $A^{(k)}_{\alpha}$) will be considered
as matrices representing cochains in $\mathcal{E}nd(\EEE)\otimes\Omega^1 (D)$ (resp.  $\mathcal{E}nd(\EEE_k)\otimes\Omega^1 (D))$
in the basis $e_{\alpha}$ over $U_{\alpha}$.
Now write (\ref{equation 22})  for $k=0$ :
\begin{equation}\label{equation 24}
K_{\alpha \beta,1}=\ud G_{\alpha \beta,1}- G_{\alpha \beta,1}A_{\beta, 0}+A_{\alpha,0} G_{\alpha \beta, 1};
\end{equation} we take into account that $I^{(1)}=J^{(1)}=0$ and that
$\ud G_{\alpha \beta,0}- G_{\alpha \beta,0}A_{\beta, 0}+A_{\alpha,0} G_{\alpha \beta, 0}=0$, the latter equation being 
a form of (\ref{equation 6}) in which $G_{\alpha \beta,0}$ are considered as matrices of endomorphisms of $\EEE$
written with respect to two bases: $e_{\alpha}$ for the source, $e_{\beta}$ for the target, and $(\ud G_{\alpha \beta,0})$
is a cocycle representing $\At^{D}(\EEE)$.

The r.h.s of (\ref{equation 24}), with the same convention that $G_{\alpha \beta,1}$ are matrices of endormorphisms of $\EEE$
with respect to the two bases, is just the cochain $(\ud a_{\alpha \beta}+[A_{\alpha}, a_{\alpha \beta}])\in\check{C}^1(\mathfrak U, \mathcal{E}nd(\EEE)\otimes\Omega^1 (D))$. As in (\ref{equation 16}), we can rewrite it as $\nabla_{\mathcal{E}nd(\EEE)}(a)$, where $a=(G_{\alpha \beta,1})$, and this representation makes obvious that $(K_{\alpha \beta,1})$ is 
a $1$-cocycle.
The differential $d_1$ of the spectral sequence being induced by $\nabla_{\mathcal{E}nd(\EEE)}$, we see that the
cocycle $(K_{\alpha \beta,1})$ is a coboundary  if and only if $$[a]=[G_{\alpha \beta,1}]\in\ker(H^{1}(X,\mathcal{E}nd(\EEE))\otimes (W'')^*\rar H^{1}(X,\mathcal{E}nd(\EEE)\otimes\Omega^1 (D))\otimes (W'')^*).$$ Assuming that $(K_{\alpha \beta,1})$ is a coboundary, we choose
$(A_{\alpha,1})$ as a solution to 
\begin{equation}\label{equation 25}
\tilde{K}_{\alpha \beta,1}= G_{\alpha \beta,0}A_{\beta,1}-A_{\alpha,1} G_{\alpha \beta,0}
\end{equation} 
Such a solution can be chosen as a linear form in $s_1,\dots,s_{N''}$. Single out one such solution and denote it 
$(A^{''}_{\alpha,1})=(A^{''}_{\alpha,1}(s_1,\dots,s_{N''}))$. Let $(A'^{(i)}_{\alpha,1}),i=1,\dots,N'$ be a basis 
of $H^{0}(\mathfrak U, \mathcal{E}nd(\EEE)\otimes\Omega^1 (D))$ dual to the coordinates $t_1,\dots,t_{N'}$ on $W'$.
Then set
$$A_{\alpha,1}=A^{''}_{\alpha,1}(s_1,\dots,s_{N''})+\sum_{i=1}^{N'}A'^{(i)}_{\alpha,1}t_i.$$
Now assume that the forms (\ref{eqnarray.6..}) have been constructed up to degree $k\geq 0$ and define them for degree $k+1$.
Start by $F_{\alpha \beta \gamma,k+1}$, which we define, as in the proof of Theorem \ref{theorem 1}, to be a lift
to $\check{Z}^{2}(\mathfrak U, \mathcal{E}nd(\EEE))\otimes k[s_1,\dots,s_{N''}]$, of the homogeneous component
of degree $k+1$ in $G^{(k)}_{\alpha \beta}G^{(k)}_{\beta \gamma}G^{(k)}_{\gamma \alpha}$, which is a cocycle
modulo $\bar{I}^{(k+1)}+(s_1,\dots,s_{N''})^{k+1}$ by the proof of Lemma \ref{lemma 1}.

Then we set $\bar{f}_{k+1}$ equal to any lift of the cohomology class $(F_{\alpha \beta \gamma,k+1})\in
H^{2}(X,\mathcal{E}nd(\EEE))\otimes k[[s_1,\dots,s_{N''}]]/\bar{I}^{(k+1)}$ to $H^{2}(X,\mathcal{E}nd(\EEE))\otimes k[s_1,\dots,s_{N''}]$.
By construction, $(F_{\alpha \beta \gamma,k+1})$ is a coboundary modulo $\bar{I}^{(k+2)}+(s_1,\dots,s_{N''})^{k+2}$,
so there exists a cochain in $$\check{C}^1(\mathfrak U,\mathcal{E}nd(\EEE))\otimes k[s_1,\dots,s_{N''}]/(\bar{I}^{(k+2)}+(s_1,\dots,s_{N''})^{k+2})$$
whose coboundary is $(F_{\alpha \beta \gamma,k+1})\mod (\bar{I}^{(k+2)}+(s_1,\dots,s_{N''})^{k+2})$, and
$(G_{\alpha \beta,k+1})$ is defined as any lift of this cochain to $\check{C}^1(\mathfrak U,\mathcal{E}nd(\EEE))\otimes k[s_1,\dots,s_{N''}]$
which is homogeneous of degree $k+1$ in $s_1,\dots,s_{N''}$.
Consider now the expression\\
\begin{eqnarray*}\tilde{K}_{\alpha \beta,k+1}=\ud G^{(k+1)}_{\alpha \beta}- G^{(k+1)}_{\alpha \beta}A^{(k)}_{\beta}+A^{(k)}_{\alpha} G^{(k+1)}_{\alpha \beta}=\ud G^{(k)}_{\alpha \beta}- G^{(k)}_{\alpha \beta}A^{(k-1)}_{\beta}+A^{(k-1)}_{\alpha} G^{(k)}_{\alpha \beta}&\\ \nonumber+\ud G_{\alpha \beta,k+1}- G_{\alpha \beta, k+1}A^{(k-1)}_{\beta}+A^{(k-1)}_{\alpha}G_{\alpha \beta,k+1}- G^{(k+1)}_{\alpha \beta}A_{\beta,k}+A_{\alpha,k}G^{(k+1)}_{\alpha \beta},\end{eqnarray*}
By the induction hypothesis, $\tilde{K}_{\alpha \beta, k}=\ud G^{(k)}_{\alpha \beta}- G^{(k)}_{\alpha \beta}A^{(k-1)}_{\beta}+A^{(k-1)}_{\alpha} G^{(k+1)}_{\alpha \beta}$ is a cocycle modulo $J^{k}+\bar{I}^{(k+1)}$ and is a coboundary modulo $J^{k+1}+\bar{I}^{(k+1)}+(t_1,\dots,t_N)^{k+1}.$
From (\ref{equation 22}), in order that $\tilde{K}_{\alpha \beta,k+1}$ has no homogeneous components of order $<k+1$ modulo $J^{k+1}+\bar{I}^{(k+1)}+(t_1,\dots,t_N)^{k+1}$, we have to set $(A_{\alpha,k})$ to be a solution of \begin{equation}\label{equation s1}
G^{(k+1)}_{\alpha\beta}A_{\beta,k}
-A_{\alpha,k}G^{(k+1)}_{\alpha\beta}\equiv\tilde{K}_{\alpha\beta,k}\mod(J^{k+1}+\bar{I}^{k+1}+(t_1,\dots,t_N)^{k+1}),\end{equation}
where $G^{(k+1)}_{\alpha \beta}$ can be replaced by $G_{\alpha \beta,0}$, so that (\ref{equation s1}) is an equation for
the cochain $(G_{\alpha \beta,0}A_{\beta,k})$ with values in $\mathcal{E}nd(\EEE)\otimes\Omega^{1}(D)$.
%Then we can set \begin{equation}\label{equation s}
%K_{\alpha\beta,k+1}=\ud G_{\alpha\beta,k+1}-G_{\alpha\beta,k+1}A_{\beta,0}+A_{\alpha,0}G_{\alpha\beta,k+1},
%\end{equation}
Thus we come to the following inductive procedure:
define $K_{\alpha \beta,k+1}$ as the homogeneous form of degree $k+1$ in $\tilde{K}_{\alpha \beta,k+1}$.
%, or give it by formula (\ref{equation s}).
Assuming it is a cocycle modulo $(J^{k+1}+\bar{I}^{(k+2)})$, we define $\kappa_{k+1}$ as a lift to 
$H^1(X,\mathcal{E}nd(\EEE)\otimes\Omega^{1}(D))\otimes k[t_1,\dots,t_N]$ of the cohomology class
$[(K_{\alpha \beta,k+1})\mod J^{k+1}+\bar{I}^{(k+2)}]$. Then $J^{(k+2)}$ is well-defined and
$(K_{\alpha \beta,k+1})$ becomes a coboundary modulo $J^{(k+2)}+\bar{I}^{(k+2)}+(t_1,\dots,t_N)^{k+2}$.
Hence we can construct $(A_{\alpha,k+1})$ as a lift to $\check{C}^0(\mathfrak U, \mathcal{E}nd(\EEE)\otimes\Omega^{1}(D))\otimes k[t_1,\dots,t_N]$
of a solution  $(A_{\alpha,k+1})$ of the equation 
$$G_{\alpha \beta,0}A_{\beta,k+1}-A_{\alpha,k+1}G_{\alpha \beta,0}\equiv\tilde{K}_{\alpha \beta,k+1}\mod
(J^{k+2}+\bar{I}^{(k+2)}+(t_1,\dots,t_N)^{k+2}).$$
Thus we have to verify that $(K_{\alpha \beta,k+1})$ is a cocycle.
\end{proof}
\begin{lemma}
$(K_{\alpha \beta,k+1})$ defined as the homogeneous component of degree $k+1$ of
$\tilde{K}_{\alpha \beta,k+1}$, is a $1$-cocycle modulo $J^{k+1}+\bar{I}^{(k+2)}.$
\end{lemma}
\begin{proof}
By the induction hypothesis, we have\\
$$\ud G^{(k)}_{\alpha \beta}\equiv G^{(k)}_{\alpha \beta}A^{(k-1)}_{\beta}-A^{(k-1)}_{\alpha} G^{(k)}_{\alpha \beta} \mod (J^{k}+\bar{I}^{(k+1)}),$$
$$G^{(k)}_{\alpha \beta}G^{(k)}_{\beta \gamma}G^{(k)}_{\gamma \alpha}\equiv 1+F_{\alpha \beta \gamma, k+1}\mod \bar{I}^{(k+1)},$$
and by construction,
\begin{eqnarray*}G_{\alpha\beta,k+1}G^{(k)}_{\beta\gamma}G^{(k)}_{\gamma\alpha}+G^{(k)}_{\alpha\beta}G_{\beta\gamma,k+1}G^{(k)}_{\gamma\alpha}+ G^{(k)}_{\alpha\beta}G^{(k)}_{\beta\gamma}G_{\gamma\alpha,k+1}&\ \\ \equiv-F_{\alpha\beta\gamma,k+1}\mod(\bar{I}^{(k+2)}+(s_1,\dots,s_{N''})^{k+2}),\end{eqnarray*}
$$K_{\alpha \beta,k+1}\equiv\ud G^{(k+1)}_{\alpha \beta}- G^{(k+1)}_{\alpha \beta}A^{(k)}_{\beta}+A^{(k)}_{\alpha} G^{(k+1)}_{\alpha \beta} \mod (J^{k+1}+\bar{I}^{(k+1)}).$$
Denote $G^{(k+1)}_{\alpha \beta},G^{(k)}_{\alpha \beta},G_{\alpha \beta,k+1},A^{(k)}_{\alpha},K_{\alpha \beta,k+1}$ by 
$G_{\alpha \beta},G'_{\alpha \beta}, G''_{\alpha \beta},A_{\alpha},K_{\alpha \beta}$ respectively.\\ We have 
\arraycolsep=0.5ex
\begin{gather}\label{gather.7..}
K_{\alpha \beta}G_{\beta \gamma}G_{\gamma \alpha}+G_{\alpha \beta}K_{\beta \gamma}G_{\gamma \alpha}+G_{\alpha \beta}G_{\beta \gamma}K_{\gamma \alpha}\equiv\ud G_{\alpha \beta}G_{\beta \gamma}G_{\gamma \alpha}+G_{\alpha \beta}\ud G_{\beta \gamma}G_{\gamma \alpha}\\ \nonumber+G_{\alpha \beta}G_{\beta \gamma}\ud G_{\gamma \alpha}-G_{\alpha \beta}A_{\beta}G_{\beta \gamma}G_{\gamma \alpha}+A_{\alpha}G_{\alpha \beta}G_{\beta \gamma}G_{\gamma \alpha}-G_{\alpha \beta}G_{\beta \gamma}A_{\gamma}G_{\gamma \alpha}+G_{\alpha \beta}A_{\beta}G_{\beta \gamma}G_{\gamma \alpha}\\ \nonumber- G_{\alpha \beta}\ud G_{\beta \gamma}G_{\gamma \alpha}A_{\alpha}+
G_{\alpha \beta}\ud G_{\beta \gamma}A_{\gamma}G_{\gamma \alpha}\equiv\ud G'_{\alpha \beta}G'_{\beta \gamma}G'_{\gamma \alpha}+G'_{\alpha \beta}\ud G'_{\beta \gamma}G'_{\gamma \alpha}+G'_{\alpha \beta}G'_{\beta \gamma}\ud G'_{\gamma \alpha}\\ \nonumber+\ud G''_{\alpha \beta}G'_{\beta \gamma}G'_{\gamma \alpha}+G'_{\alpha \beta}\ud G''_{\beta \gamma}G'_{\gamma \alpha}+G'_{\alpha \beta}G'_{\beta \gamma}\ud G''_{\gamma \alpha}+\ud G'_{\alpha \beta}G''_{\beta \gamma}G'_{\gamma \alpha}+\ud G'_{\alpha \beta}G'_{\beta \gamma}G''_{\gamma \alpha}\\ \nonumber+G''_{\alpha \beta}\ud G'_{\beta \gamma} G'_{\gamma \alpha}+G'_{\alpha \beta}\ud G'_{\beta \gamma}G''_{\gamma \alpha}+G''_{\alpha \beta}G'_{\beta \gamma}\ud G'_{\gamma \alpha}\equiv\ud(G'_{\alpha \beta}G'_{\beta \gamma}G'_{\gamma \alpha})-G'_{\alpha \beta}G''_{\beta \gamma}\ud G'_{\gamma \alpha}\\ \nonumber+\ud(G''_{\alpha \beta} G'_{\beta \gamma} G'_{\gamma \alpha}+G'_{\alpha \beta}G''_{\beta \gamma}G'_{\gamma \alpha}+G'_{\alpha \beta}G'_{\beta \gamma}G''_{\gamma \alpha})
\equiv\ud (F_{\alpha \beta \gamma,k+1})-\ud (F_{\alpha \beta \gamma,k+1})
\equiv 0\\ \nonumber \mod (J^{k+1}+\bar{I}^{(k+2)})
\end{gather}
This ends the proof.
\end{proof} 
Coming back to the proof of the Theorem, we define $f_k$ as any lift to $\HH^{2}(\CCC^{\fatdot})\otimes k[t_1,\dots,t_N]$,
homogeneous of degree $k$ in $t_1,\dots,t_N$, of the cohomology class of the cochain
\begin{equation}\label{equation 27}((K_{\alpha \beta,k}), (F_{\alpha \beta \gamma,k}))\mod (J^k +\bar{I}^{k+1})\in\check{C}^{2}(\mathfrak U,\CCC^{\fatdot})\otimes k[[t_1,\dots,t_N]]/(J^k +\bar{I}^{(k+1)}),\end{equation} which we are assuming to be a cocycle.
Then quotienting by $I$ makes (\ref{equation 27}) a coboundary of $((A_{\alpha,k}), (G_{\alpha \beta,k}))$, and the pair
$(G^{(k)}_{\alpha \beta},(A^{(k)}_{\alpha}))$ defines $(\EEE_k,\nabla_k)$ over $X\times V_k$.
It remains to prove that (\ref{equation 27}) is a cocycle with values in $\CCC^{\fatdot}\otimes k[t_1,\dots,t_N]/(J^k +\bar{I}^{k+1})$.
One part of this, namely, the equation
$$\check{d}(K_{\alpha \beta,k})=\nabla_{\mathcal{E}nd(\EEE)}(F_{\alpha \beta \gamma,k})$$
is verified by the computation (\ref{gather.7..}). The second part $\check{d}(F_{\alpha \beta \gamma,k})=0$ is guaranteed by Lemma \ref{lemma 1}.
\section{Integrable connections}\label{Int connec}
\subsection{Higher order deformations of integrable connections}
From now on, we take into account the fact that $(\EEE,\nabla)$ is an integrable connection with fixed divisor 
of poles $D$ and consider deformations of $(\EEE,\nabla)$ preserving the integrability and the divisor of poles.
In Theorem \ref{theorem 1.2}, we characterized the first order deformations of $(\EEE,\nabla)$ in terms
of the hypercohomology group $\HH^{1}(X, \FFF^{\fatdot})=\HH^{1}(X, \RRR^{\fatdot}).$
Now we will consider the second order deformation and respectively the first obstruction.
So, we search for the extension 
\begin{eqnarray}\label{eqnarray.8..}%
\tilde{g}_{\alpha \beta}=(1+\epsilon a_{\alpha \beta,1}+\epsilon^2 a_{\alpha \beta,2})g_{\alpha \beta} & \\ \nonumber
\tilde{A}_{\alpha}=A_{\alpha}+\epsilon A_{\alpha,1}+\epsilon^2 A_{\alpha,2}  & \\ \nonumber
\end{eqnarray}
of $(g_{\alpha \beta},A_{\alpha})$ to $V=\Spec k[\epsilon]/(\epsilon^3)$.
To order $1$, we have the conditions (\ref{equation 17}):
\begin{equation}
\check{d}(a_{\alpha \beta,1})=0, \check{d}(A_{\alpha,1})=\nabla(a_{\alpha \beta,1}), \nabla(A_{\alpha,1})=0.
\end{equation}
Expanding (\ref{equation 7}) to order $2$, we obtain in addition to (\ref{equation a}) and $(\ref{equation h})$, the equation
\begin{equation}\label{equation 28}
\nabla A_{\alpha,2}=-A_{\alpha,1}\wedge A_{\alpha,1},
\end{equation}
Note that $\nabla(A_{\alpha,1})=0$ implies that $\nabla(A_{\alpha,1}\wedge A_{\alpha,1})=0$.
One easily verifies the following relations
\begin{eqnarray}\label{eqnarray.9..}%
\nonumber\nabla(A_{\alpha,1}\wedge A_{\alpha,1})=0 & \\ \nonumber
\check{d}(A_{\alpha,1}\wedge A_{\alpha,1})=-\nabla(A_{\alpha,1}a_{\alpha \beta,1}-a_{\alpha \beta,1}A_{\beta,1})  & \\ \nonumber
\check{d}(A_{\alpha,1}a_{\alpha \beta,1}-a_{\alpha \beta,1}A_{\beta,1}) =\nabla(a_{\alpha \beta,1}a_{\beta \gamma,1}\circlearrowleft),
& \\ \nonumber
\end{eqnarray}
where $\circlearrowleft$ denotes the skew-symmetrization on the subscripts $\alpha,\beta,\gamma$.
These three equations express the fact that the triple 
$$((a_{\alpha \beta,1}a_{\beta \gamma,1}\circlearrowleft),(A_{\alpha,1}a_{\alpha \beta,1}-a_{\alpha \beta,1}A_{\beta,1}),
(A_{\alpha,1}\wedge A_{\alpha,1}))\in\check{C}^{2}(\mathfrak U,\RRR^{\fatdot})$$
is a cocycle with respect to the differential $D=\nabla\pm\check{d}$.
Then the conditions saying that (\ref{eqnarray.8..}) is an integrable connection with fixed divisor of poles $D$
modulo $\epsilon^3$, that is, formulas (\ref{eqnarray.4..}), (\ref{eqnarray.5..}) and (\ref{equation 28}), mean that the cocycle defined above is the coboundary of the cochain $((a_{\alpha \beta,2}),(\AAA_{\alpha,2}))$:$$D(a_2,\AAA_2)=
((a_{\alpha \beta,1}a_{\beta \gamma,1}\circlearrowleft),(A_{\alpha,1}a_{\alpha \beta,1}-a_{\alpha \beta,1}A_{\beta,1}),
(A_{\alpha,1}\wedge A_{\alpha,1})).$$
As the cocycle (\ref{eqnarray.9..}) represents the Yoneda square of $[a_1,\AAA_1]$, we deduce:
\begin{proposition}
The first order deformation $(\EEE_1,\nabla_1)$ of $(\EEE,\nabla)$ defined by the cocycle $((a_{\alpha \beta,1}), (\AAA_{\alpha,1}))$
extend to an integrable connection $(\EEE_2,\nabla_2)$ over $X\times V_2$ with fixed divisor of poles $D$ if and
only if the Yoneda square $[a_1,\AAA_1]\circ [a_1,\AAA_1]$ is 
zero in $\HH^{2}(\RRR^{\fatdot})$.
\end{proposition}

Thus the integrable case looks similar to the non-integrable one (compare to Prop \ref{proposition 2}), provided 
we replace the $2$-term complex $\CCC^{\fatdot}$ by $\RRR^{\fatdot}$. As far as only the hypercohomology
$\HH^1$ and $\HH^2$ are concerned, we can also truncate $\RRR^{\fatdot}$ at the level $2$:
$\HH^i(\RRR^{\fatdot})=\HH^i(\tilde{\RRR}^{\fatdot})$, for $i=0,1,2$, where
$\tilde{\RRR}^{\fatdot}=[\RRR^0\rar\RRR^1\rar\ker(\RRR^2\rar\RRR^3)].$
\subsection{Kuranishi space of integrable connections}
Now, we turn to the construction of the Kuranishi space of integrable connections with fixed divisor
of poles $D$. Its construction is completely similar to the one in the non-integrable case, so instead
of giving a proof of the next theorem, we will only supply some remarks indicating modifications that should be
brought to the proof of Theorem \ref{theorem 1.3} in order to get the proof in the integrable case. 

The spectral sequence $E^{p,q}_{1}=H^q(X,\RRR^p)$ converging to $\HH^{\fatdot}(\RRR^{\fatdot})$ is not
concentrated on two vertical strings, so here $\HH^{2}(\RRR^{\fatdot})$ has a filtration consisting of 
three nonzero summands which are subquotients of $H^0(X,\mathcal{E}nd(\EEE)\otimes\Omega^2_X(*D)),
H^1(X,\mathcal{E}nd(\EEE)\otimes\Omega^1_X(D),H^2(X,\mathcal{E}nd(\EEE)).$
Hence, we have to add to the forms (\ref{eqnarray.6..}) two more homogeneous forms of degree $k$, say
\begin{eqnarray}\label{eqnarray.10..}%
L_{\alpha,k}(t_1,\dots,t_N)\in\Gamma(U_{\alpha},\mathcal{E}nd(\EEE)\otimes\Omega^2_X(*D))\otimes k[t_1,\dots,t_N],  & \\ \nonumber
l_k(t_1,\dots,t_N)\in H^{0}(X,\mathcal{E}nd(\EEE)\otimes\Omega^2_X(*D))\otimes k[t_1,\dots,t_N],  
& \\ \nonumber
\end{eqnarray}
and modify according the conditions $(i),\dots,(vi)$ to which the forms (\ref{eqnarray.6..}),(\ref{eqnarray.10..}) should satisfy.
Remark also that the long exact cohomology sequence for $\CCC^{\fatdot}$ introduced in the proof of Theorem \ref{theorem 1.3} remains exact only in its $4$ terms when
$\CCC^{\fatdot}$ is replaced by $\RRR^{\fatdot}$.
\begin{theorem}\label{theorem 1.4}
Let $X$ be a complete scheme of finite type over $k$ or a complex space (in which case $k=\CC$), $\nabla$ an integrable connection on $\EEE$ with fixed divisor of poles $D$, $\RRR^{\fatdot}$
the complex of sheaves on $X$ defined above, $W=\HH^{1}(X,\RRR^{\fatdot})$, $(\delta_1\dots,\delta_N)$ a basis of $W$ and $(t_1,\dots,t_N)$ the dual coordinates on $W$.
Let $W_k$ denote the $k$-th infinitesimal neighborhood of $0$ in $W$, and $(\EEE_1, \nabla_1)$ the universal first
deformation of $(\EEE,\nabla)$ over $X\times W_1$ in the class of integrable connections with fixed divisor of poles $D$. Then there exists a formal power series
$$f(t_1,\dots,t_N)=\sum_{k=2}^{\infty} f_{k}(t_1\dots,t_N)\in \HH^{2}(X, \RRR^{\fatdot})[[t_1,\dots,t_N]],$$
where $f_k$ is  homogeneous of degree $k$ ($k\geq 2$), with the following property.
Let $I$ be the ideal of $k[[t_1,\dots,t_N]]$ generated by the image of the map $f^*:\HH^{2}(X,\RRR^{\fatdot})^{*}\rar k[[t_1,\dots,t_N]]$, adjoint to $f$.
Then for any $k\geq 2$, the pair $(\EEE_1,\nabla_1)$ extends to an integrable connection $(\EEE_k,\nabla_k)$ on $X\times V_k$, where
$V_k$ is the closed subscheme of $W_k$ defined by the ideal $I\otimes k[[t_1,\dots,t_N]]/(t_1,\dots,t_N)^{k+1}.$
\end{theorem}

\begin{remark}
The complex $\RRR^{\fatdot}$ may be replaced by its subcomplex
$0\rar\mathcal{E}nd(\EEE)\rar\mathcal{E}nd(\EEE)\otimes\Omega^1_X(D)\rar\mathcal{E}nd(\EEE)\otimes\Omega^2_X(2D)\rar\dots.$
Theorem \ref{theorem 1.3} will remain valid if we replace $\RRR^{\fatdot}$ in its statement by this smaller complex.
\end{remark}
In the case where $\nabla$ is an integrable logarithmic connection, we can reduce $\RRR^{\fatdot}$ further to
$\LLL^{\fatdot}=[0\rar\mathcal{E}nd(\EEE)\rar\mathcal{E}nd(\EEE)\otimes\Omega^1_X(\log D)\rar\mathcal{E}nd(\EEE)\otimes\Omega^2_X(\log D)\rar\dots].$
We now go over to integrable logarithmic connections. 
\subsection{Integrable logarithmic connections}
\begin{definition}
Let $X$ be a nonsingular complex projective variety, $S$ a normal crossing divisor with
smooth components. An integrable logarithmic connection $E$ on $X$ is a pair $(\EEE, \nabla)$ where $\EEE$ is a torsion free coherent sheaf of $\OOO_X$-modules on $X$ and $\nabla:\EEE\rar\EEE\otimes\Omega^{1}_X(\log S)$ is  $\CC$-linear and satisfies the Leibniz rule and the integrability condition $\nabla^2=0$ (see in the beginning of Sect. \ref{Connections}).\end{definition} 

Let $\DDD_X$ be the sheaf of algebraic differential operators on $X$ and let $\DDD_X[\log S]$ be the $\OOO_X$-subalgebra generated
by the germs of tangent vector fields which preserve the ideal sheaf of the reduced scheme $S$. 
According to \cite{Ni}, a logarithmic connection on $X$ with singularities over $S$ can be interpreted as a $\DDD_X[\log S]$-module which is coherent and torsion free as an $\OOO_X$-module.

\begin{remark}
A nonsingular integrable connection on $X$ is simply a $\DDD_X$-module which is coherent as an $\OOO_X$-module.
\end{remark}
 
\begin{definition}
An infinitesimal deformation of an integrable logarithmic connection $\EEE$ is a pair $(\EEE_V, \alpha)$, where $\EEE_V$ is a 
family of logarithmic connections parameterized by $V=\Spec(\CC[\epsilon])/\epsilon^{2}$, with an 
isomorphism $\alpha:{\EEE_V}/{\epsilon\EEE_V}\rar\EEE$.  
\end{definition}
We define $T_{\EEE}$ as the set of all equivalence classes of infinitesimal deformations of $\EEE$.
Let the sheaf $\KKK_{\EEE}$ be the kernel of $\nabla:\mathcal{E}nd(\EEE)\otimes\Omega^{1}(\log S) \rar\mathcal{E}nd(\EEE)\otimes\Omega^{2}(\log S).$ 
As the curvature of $\nabla$ is $0$, the image of $\nabla:\EEE\rar\EEE\otimes\Omega^{1}(\log S)$,
is contained in $\KKK_{\EEE}$. If $A\in H^{0}(X, \KKK_{\EEE})$, then $\nabla+\epsilon A$ is a family of logarithmic connections on 
the underlying sheaf $\EEE$ parameterized by $V$. This gives a linear map $p: H^{0}(X,\KKK_{\EEE})\rar T_{\EEE}.$
\begin{theorem}\label{theo}
If an integrable logarithmic connection $\EEE$ is locally free, the vector space $T_{\EEE}$ of infinitesimal deformations
of $\EEE$ (which equals the tangent space at $[\EEE]$ to the moduli scheme $\MMM$ of stable integrable logarithmic connections
when $\EEE$ is stable) is canonically isomorphic to the first hypercohomology $\HH^{1}(\CCC_{\EEE})$ of the complex
$\CCC_{\EEE}=(\nabla:\mathcal{E}nd(\EEE)\rar\KKK_{\EEE})$, which is in turn equal to the first hypercohomology of the logarithmic de Rham
complex $\LLL^{\fatdot}=(\mathcal{E}nd(\EEE)\otimes\Omega^{\fatdot}_{X}(\log S), \nabla)$ associated to $\mathcal{E}nd(\EEE)$.
\end{theorem}
\begin{proof}
See \cite{Ni}.
\end{proof}

We deduce the construction of the Kuranishi space of integrable logarithmic connections over $X$.
\subsection{Kuranishi space of integrable logarithmic connections}
\begin{theorem}
Let $X$ be a smooth projective variety over an algebraically closed field $k$ (or on $\CC$), $\EEE$ a vector bundle on $X$, $\nabla$ an integrable logarithmic connection on $\EEE$, $\LLL^{\fatdot}$ the complex of sheaves on $X$ defined in Theorem \ref{theo}, $W=\HH^{1}(X,\LLL^{\fatdot})$, $(\delta_1\dots,\delta_N)$ a basis of $W$ and $(t_1,\dots,t_N)$ the dual coordinates on $W$.
Let $W_k$ denote the $k$-th infinitesimal neighborhood of $0$ in $W$, and $(\EEE_1, \nabla_1)$ the universal first order
deformation of $(\EEE,\nabla)$ over $X\times W_1$ in the class of integrable logarithmic connections with fixed divisor of poles $D$. Then there exists a formal power series
$$f(t_1,\dots,t_N)=\sum_{k=2}^{\infty} f_{k}(t_1\dots,t_N)\in \HH^{2}(X,\LLL^{\fatdot})[[t_1,\dots,t_N]],$$
where $f_k$ is  homogeneous of degree $k$ ($k\geq 2$), with the following property.
Let $I$ be the ideal of $k[[t_1,\dots,t_N]]$ generated by the image of the map $f^*:\HH^{2}(X,\LLL^{\fatdot})^{*}\rar k[[t_1,\dots,t_N]]$, adjoint to $f$. Then for any $k\geq 2$, the pair $(\EEE_1,\nabla_1)$ extends to an integrable logarithmic connection $(\EEE_k,\nabla_k)$ on $X\times V_k$, where
$V_k$ is the closed subscheme of $W_k$ defined by the ideal $I\otimes k[[t_1,\dots,t_N]]/(t_1,\dots,t_N)^{k+1}.$
\end{theorem}

\section{Parabolic connections}\label{Par connec}

Let $X$ be a smooth projective curve of genus $g$.
We set
\[
 T_n :=\left.\left\{ (t_1,\ldots,t_n)\in\overbrace{X\times\cdots\times X}^n
 \right| \text{$t_i\neq t_j$ for $i\neq j$} \right\}
\]
for a positive integer $n$.
For integers $d,r$ with $r>0$, we set
\[
 \Lambda^{(n)}_r(d):=
 \left\{ (\lambda^{(i)}_j)^{1\leq i\leq n}_{0\leq j\leq r-1}
 \in \CC^{nr}
 \left|
 d+\sum_{i,j}\lambda^{(i)}_j=0
 \right\}\right..
\]
Take an element $t=(t_1,\ldots,t_n)\in T_n$ and
$\lambda=(\lambda^{(i)}_j)_{1\leq i\leq n,0\leq j\leq r-1}
\in\Lambda^{(n)}_r(d)$.
\begin{definition}\rm $(E,\nabla,\{l^{(i)}_*\}_{1\leq i\leq n})$ is said to be
a $(t,\lambda)$-parabolic connection of rank $r$ if \\
$(1)$ $E$ is a rank $r$ algebraic vector bundle on $X$, and \\
$(2)$ $\nabla: E \ra E\otimes\Omega_C^1 (\log(t_1+\dots+t_n)$ is a connection, and\\
$(3)$ for each $t_i$, $l^{(i)}_*$ is a filtration of $E|_{t_i}=l^{(i)}_0\supset l^{(i)}_1
 \supset\cdots\supset l^{(i)}_{r-1}\supset l^{(i)}_r=0$
 such that $\dim(l^{(i)}_j/l^{(i)}_{j+1})=1$ and
 $(\Res_{t_i}(\nabla)-\lambda^{(i)}_j\mathrm{id}_{E|_{t_i}})
 (l^{(i)}_j)\subset l^{(i)}_{j+1}$
 for $j=0,\ldots,r-1$.
\end{definition}
\begin{remark}\rm
By condition (3) above and \cite {EV-1}, we have
\[
 \deg E=\deg(\det(E))=-\sum_{i=1}^n\Tr\Res_{t_i}(\nabla)
 =-\sum_{i=1}^n\sum_{j=0}^{r-1}\lambda^{(i)}_j=d.
\]
\end{remark}

Let $T$ be a smooth algebraic scheme which is a covering 
of the moduli stack of $n$-pointed smooth projective curves of genus $g$
over $\CC$ and take a universal family $(\CCC,\tilde{t}_1,\ldots,\tilde{t}_n)$
over $T$.

\begin{definition}\rm
We denote the pull-back of $\CCC$ and $\tilde{t}$ with respect to the morphism
$T\times\Lambda^{(n)}_r(d)\rightarrow T$
by the same characters $\CCC$ and
$\tilde{t}=(\tilde{t}_1,\ldots,\tilde{t}_n)$.
Then $D(\tilde{t}):=\tilde{t}_1+\cdots+\tilde{t}_n$
becomes a family of  Cartier divisors on $\CCC$ flat over
$T\times\Lambda^{(n)}_r(d)$.
We also denote by $\tilde{\lambda}$ the pull-back of the
universal family on $\Lambda^{(n)}_r(d)$ by the morphism
$T\times\Lambda^{(n)}_r(d)\rightarrow \Lambda^{(n)}_r(d)$.
We define a functor
$\MMM^{\boldsymbol{\alpha}}_{\CCC/T}
(\tilde{t},r,d)$
from the category of locally noetherian schemes over
$T\times\Lambda^{(n)}_r(d)$ to the category of sets by
\[
 \MMM^{\boldsymbol{\alpha}}_{\CCC/T}
(\tilde{t},r,d)(S):=
 \left\{ (E,\nabla,\{l^{(i)}_j\}) \right\}/\sim,
\] where
\begin{enumerate}
\item $E$ is a vector bundle on $\CCC_S=\CCC\times_{T\times\Lambda^{(n)}_r(d)} S$ of rank $r$,
\item $\nabla:E\rightarrow E\otimes\Omega^1_{\CCC_S/S}(D(\tilde{t})_S)$
 is a relative connection,
\item $E|_{(\tilde{t}_i)_S}=l^{(i)}_0\supset l^{(i)}_1
 \supset\cdots\supset l^{(i)}_{r-1}\supset l^{(i)}_r=0$
 is a filtration by subbundles such that
 $(\Res_{(\tilde{t}_i)_S}(\nabla)-(\tilde{\lambda}^{(i)}_j)_S)(l^{(i)}_j)
 \subset l^{(i)}_{j+1}$
 for $0\leq j\leq r-1$, $i=1,\ldots,n$,
\item for any geometric point $s\in S$,
$\dim (l^{(i)}_j/l^{(i)}_{j+1})\otimes k(s)=1$ for any $i,j$ and
$(E,\nabla,\{l^{(i)}_j\})\otimes k(s)$
is $\alpha$-stable.
\end{enumerate}
Here $(E,\nabla,\{l^{(i)}_j\})\sim
(E',\nabla',\{l'^{(i)}_j\})$ if
there exist a line bundle $\LLL$ on $S$ and
an isomorphism $\sigma:E\stackrel{\sim}\ra E'\otimes\LLL$ 
such that $\sigma|_{t_i}(l^{(i)}_j)=l'^{(i)}_j$ for any $i,j$ and the diagram
\[
 \begin{CD}
  E @>\nabla>> E\otimes\Omega^1_{\CCC/T}(D(\tilde{t})) \\
  @V \sigma VV  @V\sigma\otimes\mathrm{id} VV \\
  E'\otimes\LLL @>\nabla'>>
  E'\otimes\Omega^1_{\CCC/T}(D(\tilde{t}))\otimes\LLL \\
 \end{CD}
\]
commutes.
\end{definition}
We now can construct the moduli space of this functor.
\begin{theorem}\label{moduli-exists}
There exists a relative fine moduli scheme
\[
 M^{\boldsymbol{\alpha}}_{\CCC/T}(\tilde{t},r,d)\rightarrow T\times\Lambda^{(n)}_r(d)
\]
of $\boldsymbol{\alpha}$-stable parabolic connections of rank $r$ and degree $d$,
which is smooth, irreducible and quasi-projective and has an algebraic symplectic structure.
The fiber $M^{\boldsymbol{\alpha}}_{\CCC_x}(\tilde{t}_x,\lambda)$
over $(x,\lambda)\in T\times\Lambda^{(n)}_r(d)$ is the irreducible moduli space of
$\boldsymbol{\alpha}$-stable $(\tilde{t}_x,\lambda)$ parabolic connections
whose dimension is $2r^2(g-1)+nr(r-1)+2$ if it is non-empty.
\end{theorem}
\begin{proof}
See \cite{I}.
\end{proof}

Let $(\tilde{E},\tilde{\nabla},\{\tilde{l}^{(i)}_j\})$
be a universal family on
$\CCC\times_T M^{\boldsymbol{\alpha}}_{\CCC/T}(\tilde{t},r,d)$.
We define a complex $\GGG^{\fatdot}$ by
\begin{align*}
 \GGG^0&:=\left\{ s\in {\mathcal End}(\tilde{E}) \left|
 \text{$s|_{\tilde{t}_i\times M^{\boldsymbol{\alpha}}_{\CCC/T}(\tilde{t},r,d)}
 (\tilde{l}^{(i)}_j)\subset\tilde{l}^{(i)}_j$ for any $i,j$}
 \right\}\right. \\
 \GGG^1&:=\left\{ s\in {\mathcal End}(\tilde{E})
 \otimes\Omega^1_{\CCC/T}(D(\tilde{t})) \left|
 \text{$\Res_{\tilde{t}_i\times M^{\boldsymbol{\alpha}}_{\CCC/T}(\tilde{t},r,d)}(s)
 (\tilde{l}^{(i)}_j)\subset \tilde{l}^{(i)}_{j+1}$ for any $i,j$}
 \right\}\right. \\
 \nabla_{\GGG^{\fatdot}} &: \GGG^0 \lra \GGG^1 ; \quad
 \nabla_{\GGG^{\fatdot}}(s)=\tilde{\nabla}\circ s - s\circ\tilde{\nabla}.
\end{align*}

As in the previous section, we can construct the Kuranishi space of $(t,\lambda)$-parabolic connections on a smooth projective curve
in using the hypercohomology of $\GGG^{\fatdot}$.
\begin{theorem}
Let $X$ be a smooth projective curve over $k$, $(\EEE,\nabla,\{l^{(i)}_{*}\})$ a $(t,\lambda)$-parabolic connection on $X$, $\GGG^{\fatdot}$ the complex of sheaves on $X$ defined above, $W=\HH^{1}(X,\GGG^{\fatdot})$, $(\delta_1\dots,\delta_N)$ a basis of $W$ and $(t_1,\dots,t_N)$ the dual coordinates on $W$.
Let $W_k$ denote the $k$-th infinitesimal neighborhood of $0$ in $W$, and $(\EEE_1,\nabla_1,\{l^{(i)}_{*}\}_{1})$ the universal first order
deformation of $(\EEE,\nabla,\{l^{(i)}_{*}\})$ over $X\times W_1$ in the class of $(t,\lambda)$-parabolic connections. Then there exists a formal power series
$$f(t_1,\dots,t_N)=\sum_{k=2}^{\infty} f_{k}(t_1\dots,t_N)\in \HH^{2}(X, \GGG^{\fatdot})[[t_1,\dots,t_N]],$$
where $f_k$ is  homogeneous of degree $k$ ($k\geq 2$), with the following property.
Let $I$ be the ideal of $k[[t_1,\dots,t_N]]$ generated by the image of the map $f^*:\HH^{2}(X,\GGG^{\fatdot})\rar k[[t_1,\dots,t_N]]$, adjoint to $f$.
Then for any $k\geq 2$, the triple $(\EEE_1,\nabla_1,\{l^{(i)}_{*}\}_{1})$ extends to a $(t,\lambda)$-parabolic connection $(\EEE_k,\nabla_k,\{l^{(i)}_{*}\}_{k})$ on $X\times V_k$, where
$V_k$ is the closed subscheme of $W_k$ defined by the ideal $I\otimes k[[t_1,\dots,t_N]]/(t_1,\dots,t_N)^{k+1}.$
\end{theorem}
We now want to construct the Kuranishi space of $T$-parabolic bundles.
Let $T$ be a finite set of smooth points $\{P_1,\dots,P_n\}$ of $X$ and $W$ a vector bundle on $X$.
\begin{definition}
By a quasi-parabolic structure on a vector bundle $W$ at a smooth point $P$ of $X$, we mean a choice of a flag
$$W_P=F_1(W)_P\supset F_2(W)_P\supset...\supset F_l(W)_P=0,$$ in the fibre $W_P$ of $W$ at $P$.
A parabolic structure at $P$ is a pair consisting of a flag as above and a sequence $0\leq\alpha_1<\alpha_2<...<\alpha_l<1$ of weights of $W$ at $P$.
\end{definition}
The integers $k_1=\dim F_1(W)_P-\dim F_2(W)_P$,\ldots, $k_l=\dim(F_l(W)_P)$ are called the multiplicities of $\alpha_1,\dots,\alpha_l$.
A $T$-parabolic structure on $W$ is the triple consisting of a flag at $P$, some weights $\alpha_i$, and their multiplicities $k_i$.
A vector bundle $W$ endowed with a $T$-parabolic structure is called a $T$-parabolic bundle. 
\begin{definition}
A $T$-parabolic bundle $W_1$ on $X$ is a $T$-parabolic subbundle of a $T$-parabolic bundle $W_2$ on $X$, if
$W_1$ is a subbundle of $W_2$ and at each smooth point $P$ of $T$, the weights of $W_1$ are a subset of those of $W_2$.
Further, if we take the weight $\alpha_{j_0}$ such that $1\leq j_0\leq m$, and  the weight $\beta_{k_0}$ for the greatest integer $k_0$ such that $F_{j_0}(W_1)_P\subset F_{k_0}(W_2)_P$, then $\alpha_{j_0}=\beta_{k_0}$.
\end{definition}
\begin{definition}
The parabolic degree of a $T$-parabolic vector bundle $W$ on $X$ is 
$$\mathrm{par\deg}(W):=\deg(W)+\sum_{P\in I}\sum_{i=1} ^{r}k_i (P)\alpha_i (P).$$
\end{definition}
\begin{definition}
A $T$-parabolic bundle $W$
is stable (resp. semistable) if
for any proper nonzero $T$-parabolic subbundle $W'\subset W$
the inequality
\begin{gather*} {\mathrm{par\deg} W'}<{(\textrm{resp. $\leq$})}
 \frac{\mathrm{par\deg W} \rk(W')}{\rk W}\end{gather*}
holds.
\end{definition}

We have a forgetful map $g$ from $(t,\lambda)$ parabolic connections to $T$-parabolic bundles.
We thus can construct the Kuranishi space of $T$-parabolic bundles by following an analogous argument
to the one given above. We first introduce the Higgs field $\Phi:\EEE\rar\EEE\otimes\Omega_X^1(D)$ defined as follows:
$$\forall p\in X, \forall f\in\OOO_{X,p}, \forall s\in\EEE_P, \Phi(fs)=f\Phi(s).$$
We afterwards consider a parabolic bundle $\EEE$ with fixed weights and parabolic points $P_1,\dots,P_N$. We set $L=K\otimes\OOO(P_1,\dots,P_N)$, the line bundle associated to the canonical divisor together with the divisor of poles $D=P_1+\dots+P_N$.
The sheaf of rational $1$-forms on $X$ is identified with the sheaf of rational sections of the canonical bundle having single poles at points
$P_1,\dots,P_N$. We replace $t_i$ by $P_i$, for $i=1,\dots,N$ and $M^{\boldsymbol{\alpha}}_{\CCC/T}(\tilde{t},r,d)$ by $M^s_{T}$. 
We define a complex $\BBB^{\fatdot}$ by
\begin{align*}
 \BBB^0&:=\left\{ s\in {\mathcal End}(\tilde{E}) \left|
 \text{$s|_{\tilde{P}_i\times M^{s}_{Z,\CCC/T}(\tilde{P},r,d)}
 (\tilde{l}^{(i)}_j)\subset\tilde{l}^{(i)}_j$ for any $i,j$}
 \right\}\right. \\
 \BBB^1&:=\left\{ s\in {\mathcal End}(\tilde{E})
 \otimes\Omega^1_{\CCC/T}(D(\tilde{Pi})) \left|
 \text{$\Res_{\tilde{P}_i\times M^{s}_{Z,\CCC/T}(\tilde{P},r,d)}(s)
 (\tilde{l}^{(i)}_j)\subset \tilde{l}^{(i)}_{j+1}$ for any $i,j$}
 \right\}\right.\\
 \ad\Phi_{\BBB^{\fatdot}} &: \BBB^0 \lra \BBB^1 ; \quad
 \ad\Phi_{\BBB^{\fatdot}}(s)=\tilde{\Phi}\circ s - s\circ\tilde{\Phi}.
\end{align*}
From this, we deduce the construction of the Kuranishi space of $T$-parabolic bundles on a smooth projective curve.
\begin{theorem}
Let $X$ be a smooth projective curve over $k$ or a complex space (in which case $k=\CC$), $\EEE$  a $T$-parabolic bundle on $X$, $\BBB^{\fatdot}$ the complex of sheaves on $X$ defined as above, $W=\HH^{1}(X,\BBB^{\fatdot})$, $(\delta_1\dots,\delta_N)$ a basis of $W$ and $(t_1,\dots,t_N)$ the dual coordinates on $W$.
Let $W_k$ denote the $k$-th infinitesimal neighborhood of $0$ in $W$, and $\EEE_1$ the universal first order
deformation of $\EEE$ over $X\times W_1$. Then there exists a formal power series
$$f(t_1,\dots,t_N)=\sum_{k=2}^{\infty} f_{k}(t_1\dots,t_N)\in \HH^{2}(X, \BBB^{\fatdot})[[t_1,\dots,t_N]],$$
where $f_k$ is  homogeneous of degree $k$ ($k\geq 2$), with the following property.
Let $I$ be the ideal of $k[[t_1,\dots,t_N]]$ generated by the image of the map $f^*:\HH^{2}(X,\BBB^{\fatdot})^{*}\rar k[[t_1,\dots,t_N]]$, adjoint to $f$.
Then for any $k\geq 2$, $\EEE_1$ extends to a $T$-parabolic bundle $\EEE_k$ on $X\times V_k$, where
$V_k$ is the closed subscheme of $W_k$ defined by the ideal $I\otimes k[[t_1,\dots,t_N]]/(t_1,\dots,t_N)^{k+1}.$
\end{theorem} 

\subsection*{Acknowledgements}
I am greatly indebted to my former research advisor D.~Markushevich for his continous guidance and help.

\renewcommand\refname{References}

\end{document}